\documentclass[12pt]{article}
\usepackage[latin1]{inputenc}
\usepackage{amsmath,amsthm,amssymb}
\usepackage{amsfonts}
\usepackage{amsmath,amsthm,amssymb,amscd}
\usepackage{latexsym}
\usepackage{color}
\usepackage{color,enumitem,graphicx}
\usepackage[colorlinks=true,urlcolor=blue,
citecolor=red,linkcolor=blue,linktocpage,pdfpagelabels,
bookmarksnumbered,bookmarksopen]{hyperref}
\usepackage{graphicx}
\usepackage{mathrsfs}
\usepackage{bm}
\usepackage{cite}

\makeatletter \@addtoreset{equation}{section} \makeatother

\newtheorem{theorem}{Theorem}[section]

\newtheorem{proposition}{Proposition}[section]
\newtheorem{lemma}{Lemma}[section]

\numberwithin{equation}{section}

\allowdisplaybreaks

\begin{document}

\title{Normalized ground states for a fractional Choquard system in $\mathbb{R}$}
\author
{Wenjing Chen\footnote{Corresponding author.}\ \footnote{E-mail address:\, {\tt wjchen@swu.edu.cn} (W. Chen),   {\tt zxwangmath@163.com} (Z. Wang)}\  \ and    Zexi Wang\\
\footnotesize  School of Mathematics and Statistics, Southwest University,
Chongqing, 400715, P.R. China}
\date{ }
\maketitle

\begin{abstract}
{In this paper, we study the following fractional Choquard system 
\begin{align*}
  \begin{split}
  \left\{
  \begin{array}{ll}
   (-\Delta)^{1/2}u=\lambda_1 u+(I_\mu*F(u,v))F_u (u,v),
    \quad\mbox{in}\ \ \mathbb{R},  \\
    (-\Delta)^{1/2}v=\lambda_2 v+(I_\mu*F(u,v)) F_v(u,v),
    \quad\mbox{in}\ \ \mathbb{R},  \\
    \displaystyle\int_{\mathbb{R}}|u|^2\mathrm{d}x=a^2,\quad \displaystyle\int_{\mathbb{R}}|v|^2\mathrm{d}x=b^2,\quad u,v\in H^{1/2}(\mathbb{R}),\\
    \end{array}
    \right.
  \end{split}
  \end{align*}
where  $(-\Delta)^{1/2}$ denotes the $1/2$-Laplacian operator, $a,b>0$ are prescribed, $\lambda_1,\lambda_2\in \mathbb{R}$, $I_\mu(x)=\frac{{1}}{{|x|^\mu}}$ with $\mu\in(0,1)$, $F_u,F_v$ are partial derivatives of $F$ and $F_u,F_v$ have exponential critical growth in $\mathbb{R}$. By using a minimax principle and analyzing the monotonicity of the ground state energy with respect to the prescribed masses, we obtain at least one normalized ground state solution for the above system.
}

\end{abstract}

\section{{\bfseries Introduction}}\label{introduction}
This paper deals with the following system
\begin{align}\label{problem}
  \begin{split}
  \left\{
  \begin{array}{ll}
   (-\Delta)^{1/2}u=\lambda_1 u+(I_\mu*F(u,v))F_u (u,v),
    \quad\mbox{in}\ \ \mathbb{R},  \\
    (-\Delta)^{1/2}v=\lambda_2 v+(I_\mu*F(u,v)) F_v(u,v),
    \quad\mbox{in}\ \ \mathbb{R},
    \end{array}
    \right.
  \end{split}
  \end{align}
  with the prescribed masses
  \begin{equation}\label{problem'}
     \displaystyle\int_{\mathbb{R}}|u|^2\mathrm{d}x=a^2,\quad \displaystyle\int_{\mathbb{R}}|v|^2\mathrm{d}x=b^2,\quad u,v\in H^{1/2}(\mathbb{R}),
  \end{equation}
where  $(-\Delta)^{1/2}$ denotes the 1/2-Laplacian operator, $a,b>0$ are prescribed, $\lambda_1,\lambda_2\in \mathbb{R}$, $I_\mu(x)=\frac{{1}}{{|x|^\mu}}$ with $\mu\in(0,1)$, $F_u,F_v$ are partial derivatives of $F$ and $F_u,F_v$ have exponential subcritical in the sense of the Trudinger-Moser inequality, see \cite{lula,ozawa,Taka}.

The interest in studying problem \eqref{problem}-\eqref{problem'} originates from the study of the following nonlocal elliptic equation
\begin{align}\label{1.2}
(-\Delta)^su=\lambda u +(I_\mu*F(u))f(u), \quad \text{in $\mathbb{R}^N$},
\end{align}
where $s\in(0,1)$, $I_\mu=\frac{1}{|x|^\mu}$ with $\mu\in (0,N)$, $F(u)$ is the primitive function of $f(u)$, and $(-\Delta )^s$ is the fractional Laplacian operator defined by
\begin{align*}
(-\Delta )^su(x):=C({N,s})\ \mbox{P.V.}\int_{\mathbb{R}^N}\frac{u(x)-u(y)}{|x-y|^{N+2s}}\mathrm{d}y, \quad \text {in $\mathbb{R}^N$},
\end{align*}
for $u\in C_0^\infty(\mathbb{R}^N)$,
where P.V. means the Cauchy principal value and $C({N,s})$ is some positive normalization constant, we refer to \cite{di} for more details. The nonlocal equation \eqref{1.2} arises in many interesting physical situations in quantum theory and plays an important role in describing the finite-range many-body interactions. Equation \eqref{1.2} is well known in the literature as the Choquard equation and was first introduced by Penrose in \cite{P} to investigate the self-gravitational collapse of a
quantum mechanical wave function.

To get solutions of $(\ref{1.2})$, one way is to fix $\lambda\in \mathbb{R}$ and look for solutions of $(\ref{1.2})$
as critical points of the energy functional $\hat J:H^{s}(\mathbb{R}^{N})\rightarrow \mathbb{R}$ (see e.g. \cite{cle,dSS,SGY,ZW})
\begin{align*}
\hat J(u)=\frac{1}{2} \int_{\mathbb{R}^N}(|(-\Delta)^{s/2}u|^2-\lambda |u|^2
)\mathrm{d}x-\frac{1}{2} \int_{\mathbb{R}^N}(I_\mu*F(u))F(u)\mathrm{d}x
\end{align*}
with
\begin{equation*}
	\int_{\mathbb{R}^N}|(-\Delta)^{s/2}u|^2\mathrm{d}x=\int_{\mathbb{R}^N}\int_{\mathbb{R}^N}\frac{|u(x)-u(y)|^2}{|x-y|^{N+2s}}\,\mathrm{d}x\mathrm{d}y,
\end{equation*}
where $H^{s}(\mathbb{R}^{N})$ is a Hilbert space with the inner product and norm respectively
$$
\langle u,v\rangle=\int_{\mathbb{R}^{N}}(-\Delta)^{s/2}u(-\Delta)^{s/2}v\mathrm{d}x+\int_{\mathbb{R}^{N}} uv\mathrm{d}x,
$$
$$
\|u\|=\Big(\int_{\mathbb{R}^{N}}|(-\Delta)^{s/2}u|^{2}\mathrm{d}x+\int_{\mathbb{R}^{N}}|u|^{2}\mathrm{d}x\Big)^{1/2}.
$$

Another interesting way is to prescribe the $L^2$-norm of the unknown $u$, and $\lambda\in \mathbb{R}$ appears as a Lagrange multiplier, that is to consider the following problem
\begin{equation}\label{problem1}
	 \begin{cases}
   (-\Delta)^su=\lambda u +(I_\mu*F(u))f(u),\ \  \mbox{in}\ \mathbb{R}^N,\\
   \displaystyle\int_{\mathbb{R}^N}|u|^2 \mathrm{d}x=a^2,\quad u\in H^{s}(\mathbb{R}^N),
   \end{cases}
\end{equation}
for any fixed $a>0$. This type of solutions is called normalized solution, and can be obtained by
 looking for critical points of the following energy functional
\begin{align*}
\bar J(u)=\frac{1}{2} \int_{\mathbb{R}^N} |(-\Delta)^{s/2}u|^2
\mathrm{d}x-\frac{1}{2} \int_{\mathbb{R}^N}(I_\mu*F(u))F(u)\mathrm{d}x
\end{align*}
on the $L^2$-sphere $$\bar S(a):=\Big\{u\in H^s(\mathbb{R}^N):\int_{\mathbb{R}^N}|u|^2 \mathrm{d}x=a^2\Big\}.$$
In particular, we are interested in looking for ground state solutions,  i.e., solutions minimizing $\bar J$ on $\bar S(a)$ among all nontrivial solutions, and the associated energy is called ground state energy.

In recent years, there are many works dedicated to study \eqref{problem1}, see \cite{alvesji,BM,chenzou1,soave1,jeanjean,jeanjeanlu,Li1,soave,soave3,wei} for $s=1$ and $\mu=0$, \cite{bartsch,dengyu,Li2,liye,yao,yuan} for $s=1$ and $\mu\in (0,N)$,  \cite{LZ,LZ1,ZH,ZZ} for $s\in (0,1)$ and $\mu=0$, \cite{CSW,HRZ,LHXY,liluo,liluoyang,Yang} for $s\in (0,1)$ and $\mu\in (0,N)$. In particular, Jeanjean \cite{jeanjean} first showed that a normalized ground state solution does exist for the following equation when $f$ is $L^2$-supercritical growth
\begin{align}\label{jj}
  \begin{split}
  \left\{
  \begin{array}{ll}
   -\Delta u=\lambda u+f(u),
    \quad\mbox{in}\ \ \mathbb{R}^N,  \\
    \displaystyle\int_{\mathbb{R}^N}|u|^2dx=a^2,\quad u\in H^1(\mathbb{R}^N).
    \end{array}
    \right.
  \end{split}
  \end{align}
By using a minimax principle based on the homotopy stable family,
Bartsch and Soave \cite{soave1} presented a new approach that is based on a natural constraint and proved the existence of normalized solutions also in this case. Inspired by \cite{jeanjean,soave1}, Soave \cite{soave} studied \eqref{jj} with combined nonlinearities $f(u)=\omega|u|^{q-2}u+|u|^{p-2}u$, $\omega\in \mathbb{R}$, $2<q\leq 2+\frac{4}{N}\leq p<2^*$ and $q<p$, where $2^*=\infty$ if $N\leq2$ and $2^*=\frac{2N}{N-2}$ if $N\geq3$. The Sobolev critical case $p=2^*$ and $N\geq3$ was considered by Soave \cite{soave3}. In the case $N=2$ and $f$ has exponential critical growth, the existence of normalized solutions of \eqref{jj} has been discussed by Alves et al. \cite{alvesji}.
Besides, Deng and Yu \cite{dengyu}, Chen et al. \cite{CSW} studied \eqref{problem1} with $f$ having exponential critical growth when $s=1$ and $s=\frac{1}{2}$, respectively.

Considering the following system with the mass constraints
\begin{align}\label{Bose}
  \begin{split}
  \left\{
  \begin{array}{ll}
   (-\Delta)^s u=\lambda_1 u+(I_\mu*F(u,v))F_u(u,v),\quad \text{in $\mathbb{R}^N$},\\
   (-\Delta)^s v=\lambda_2 v+(I_\mu*F(u,v))F_v(u,v),\quad \text{in $\mathbb{R}^N$},\\
    \displaystyle\int_{\mathbb{R}^N}|u|^2dx=a^2,\quad \displaystyle\int_{\mathbb{R}^N}|v|^2dx=b^2,\quad u,v\in H^s(\mathbb{R}^N).\\
    \end{array}
    \right.
  \end{split}
  \end{align}
This system has an important physical significance in nonlinear optics and Bose-Einstein condensation. The most famous case is that of coupled Gross-Pitaevskii
equations in dimension $N\leq3$ with $F_u(u,v)=\mu_1|u|^{p-2}u+ r_1\tau|v|^{r_2}|u|^{r_1-2}u $, $F_v(u,v)=\mu_2|u|^{q-2}u+ r_2\tau|u|^{r_1}|v|^{r_2-2} v$, $s=1$, $\mu=0$, $p=q=4$, $r_1=r_2=2$, and $\mu_1,\mu_2,\tau>0$, which models
Bose-Einstein condensation.
The particular case in $\mathbb{R}^3$ was investigated in the companion paper \cite{bartschjean}, and has been further developed by many scholars, we refer the readers to  \cite{soave1,BS2,zhongzou,bartschjean1,chenzou2,BLZ,GJ,LiZou,LWYZ,DY,ZZR,chenwang2} and references therein. It is worth pointing out that in \cite{DY}, Deng and Yu first considered normalized solutions of \eqref{Bose} with general nonlinear terms involving exponential critical growth when $N=2$, $s=1$ and $\mu=0$.  
The authors in \cite{chenwang2} also studied \eqref{Bose} with exponential critical nonlinearities when $N=4$, $s=2$ and $\mu\in(0,4)$. In addition, by using the Adams function \cite{LY3}, they gave a more natural growth condition to estimate the upper bound of the ground state energy.

The study of normalized solutions for \eqref{Bose} is a hot topic in nonlinear PDEs nowadays. However, as far as we know, there are only a few papers dealing with such problems with general nonlinearities besides the ones already mentioned above \cite{chenwang2,DY}. Based on these facts, in this work, we focus on the existence of normalized ground state solutions of problem $(\ref{problem})$-\eqref{problem'}.

More precisely, we assume that $F$ satisfies:

$(F_1)$   For $j=1,2$, $F_{z_j}(z)\in C(\mathbb{R}\times \mathbb{R},\mathbb{R})$, and $\lim\limits_{|z|\to0}\frac{|F_{z_j}(z)|}{|z|^\kappa}=0$ for some $\kappa>2-\mu$;

$(F_2)$  $F_{z_j}(z)$ ($j=1,2$) has exponential critical growth at infinity, i.e.,
\begin{align*}
\lim\limits_{|z|\to+\infty}\frac{|F_{z_j}(z)|}{e^{\alpha |z|^2}}=
\begin{cases}
0,\quad  &\mbox{for any $\alpha>\pi$,}\\
+\infty,\quad &\mbox{for any $0<\alpha <\pi$;}
\end{cases}
\end{align*}

$(F_3)$ There exists a constant $\theta>3-\mu$ such that $0<\theta F(z)\leq z\cdot \nabla F(z)$ for all $z\neq (0,0)$;

$(F_4)$ For any $z\in \mathbb{R}\backslash \{0\}\times \mathbb{R}\backslash \{0\}$, $0< F_{z_j}(z)z_j <(2-\mu)F(z)$, $j=1,2$;

$(F_5)$ $F_{z_1}(0,z_2)\neq0$ for all $z_2\in \mathbb{R}\backslash\{0\}$ and $F_{z_2}(z_1,0)\neq0$ for all $z_1\in \mathbb{R}\backslash\{0\}$;

$(F_6)$ For any $z\in{\mathbb{R} \backslash \{0\}}\times {\mathbb{R} \backslash \{0\}}$, define $\widetilde F(z):=z\cdot \nabla F(z)-(2-\mu)F(z)$, then $\nabla \widetilde F(z)$ exists
and
\begin{equation*}
  (3-\mu)F(\hat{z})\widetilde F(\tilde{z})<F(\hat{z})\tilde{z}\cdot \nabla\widetilde F(\tilde{z})+\widetilde F(\hat{z})(\widetilde F(\tilde{z})-F(\tilde{z})),\ \ \text{for any $\hat{z}, \tilde{z}\in {\mathbb{R} \backslash \{0\}}\times {\mathbb{R} \backslash \{0\}}$};
\end{equation*}

$(F_7)$ There exists $\beta_0>0$ such that $\liminf\limits_{|z_1|,|z_2|\to+\infty}\frac{F(z)[z\cdot \nabla F(z)]}{e^{2\pi |z|^2}}\geq {\beta_0}$.

Our main result can be stated as follows:

\begin{theorem}\label{thm1.1}
Assume that $F$ satisfies $(F_1)$-$(F_7)$, then problem $(\ref{problem})$-\eqref{problem'} has at least one ground state solution.
\end{theorem}

This paper is organized as follows. Section \ref{sec preliminaries} contains some preliminaries. In Section \ref{vf}, we give the variational framework of problem \eqref{problem}-\eqref{problem'}.  Section \ref{minimax} is devoted to estimate the upper bound of the ground state energy. The monotonicity of the ground state energy with respect to the masses are studied in Section \ref{mono}. In Section \ref{ps}, we use the minimax principle to construct a bounded $(PS)$ sequence. Finally, in Section \ref{proof}, we give the proof of Theorem \ref{thm1.1}.

Throughout this paper, we will use the notation $\|\cdot\|_q:=\|\cdot\|_{L^q(\mathbb{R}^N)}$, $q\in [1,\infty]$, $B_r(x):=\{y\in \mathbb{R}:|y-x|<r\}$ is the open ball of radius $r$ around $x$, $C,C_i,i\in \mathbb{N}^+$ denote positive constants possibly different from line to line.

\section{{\bfseries Preliminaries}}\label{sec preliminaries}

In this section, we give some preliminaries.
\begin{proposition}\label{hardy}
\cite[Theorem 4.3]{Lieb} Let $1<r,t<\infty$ and $0<\mu<N$ with $\frac{1}{r}+\frac{1}{t}+\frac{\mu}{N}=2$. If $f\in L^r(\mathbb{R}^N)$ and $h\in L^t(\mathbb{R}^N)$, then there exists a sharp constant $C(N,\mu,r,t)>0$ such that
\begin{align}\label{HLS}
\int_{\mathbb{R}^N}\int_{\mathbb{R}^N}\frac{f(x)h(y)}{|x-y|^\mu}\mathrm{d}x\mathrm{d}y\leq C(N,\mu,r,t)\|f\|_r\|h\|_t.
\end{align}
\end{proposition}

\begin{lemma}\label{GN}
(The fractional Gagliardo-Nirenberg-Sobolev inequality) \cite{frank} Let $u\in H^s(\mathbb{R}^N)$ and $p\in [2,\frac{2N}{N-2s})$, then there exists a sharp constant $C(N,s,p)>0$ such that
\begin{align}\label{gns}
\int_{\mathbb{R}^N}|u|^p\mathrm{d}x\leq C(N,s,p)\Big( \int_{\mathbb{R}^N} |(-\Delta)^{\frac{s}{2}}u|^2\mathrm{d}x
\Big)^{\frac{N(p-2)}{4s}}\Big(
\int_{\mathbb{R}^N} |u|^2\mathrm{d}x
\Big)^{\frac{p}{2}-\frac{N(p-2)}{4s}}.
\end{align}
\end{lemma}

\begin{lemma}\label{tm}
(Full range Adachi-Tanaka-type on $H^{1/2}(\mathbb{R})$) \cite[Theorem 1]{Taka} It holds that
\begin{align}\label{att}
  \sup_{u\in H^{1/2}(\mathbb{R})\backslash \{0\},\|(-\Delta)^{1/4}u\|_2 \leq1}\frac{1}{\|u\|_2^2}\int_{\mathbb{R}}(e^{\alpha|u|^2}-1)\mathrm{d}x
  \begin{cases}
  <\infty, \quad \alpha<\pi,\\
  =\infty, \quad \alpha\geq\pi.
  \end{cases}
 \end{align}
\end{lemma}

\begin{lemma}\cite[Lemma 2.3]{DT}\label{alge}
Suppose that $a_1, a_2, \ldots, a_k\geq 0$ with $a_1+a_2+\cdots +a_k<1$, then there exist $p_1, p_2, \ldots, p_k>1$ satisfying $\frac{1}{p_1}+\frac{1}{p_2}+\cdots +\frac{1}{p_k}=1$ such that $p_ia_i<1$ for $i=1, 2, \ldots, k$. Moreover, if $a_1, a_2, \ldots, a_k\geq 0$ satisfying $a_1+a_2+\cdots +a_k=1$, then there exist $p_1, p_2, \ldots, p_k>1$ such that $\frac{1}{p_1}+\frac{1}{p_2}+\cdots +\frac{1}{p_k}=1$ and $p_ia_i=1$ for $i=1, 2, \ldots, k$.
\end{lemma}

\section{{\bfseries The variational framework}}\label{vf}
For the fractional Laplacian operator, the special case when $s=1/2$ is called the square of the Laplacian. We recall the definition of the fractional Sobolev space
\begin{equation*}
H^{1/2}(\mathbb{R})= \Big\{ u \in L^2(\mathbb{R}): \int_{\mathbb{R}}\int_{\mathbb{R}}\frac{|u(x)-u(y)|^2}{|x-y|^2}\mathrm{d}x\mathrm{d}y<\infty \Big\},
\end{equation*}
endowed with the standard norm
\begin{align*}
\|u\|_{1/2}=\Big(\frac{1}{2\pi}[u]_{1/2}^2+\int_{\mathbb{R}}|u|^2\mathrm{d}x\Big)^{1/2},
\end{align*}
where the term
\begin{equation*}
	[u]_{1/2}=\Big(\int_{\mathbb{R}}\int_{\mathbb{R}}\frac{|u(x)-u(y)|^2}{|x-y|^{2}}\,\mathrm{d}x\mathrm{d}y\Big)^{1/2}
\end{equation*}
denotes the Gagliardo semi-norm of a function $u$. Moreover, by \cite [Proposition 3.6]{di}, we have
\begin{align*}
\|(-\Delta)^{1/4}u\|_2^2=\frac{1}{2\pi}\int_{\mathbb{R}}\int_{\mathbb{R}}\frac{|u(x)-u(y)|^2}{|x-y|^{2}}\,\mathrm{d}x\mathrm{d}y,\ \ \mbox{for any}\  u\in H^{1/2}(\mathbb{R}).
\end{align*}

Let $\mathcal{X}:=H^{1/2}(\mathbb{R})\times H^{1/2}(\mathbb{R})$ with the norm
\begin{equation*}
  \|(u,v)\|:=\Big(\frac{1}{2\pi}[u]_{1/2}^2+\frac{1}{2\pi}[v]_{1/2}^2+\int_{\mathbb{R}}|u|^2\mathrm{d}x+\int_{\mathbb{R}}|v|^2\mathrm{d}x\Big)^{1/2}.
\end{equation*}
Moreover, for any $c>0$, we set
\begin{equation*}
  S(c):=\Big\{u\in H^{1/2}(\mathbb{R}):\int_{\mathbb{R}}|u|^2\mathrm{d}x=c^2\Big\},
\end{equation*}
and
\begin{equation*}
  \mathcal{S}:=S(a)\times S(b).
\end{equation*}

Problem \eqref{problem}-\eqref{problem'} has a variational structure and its associated energy functional $\mathcal{J}: \mathcal{X}\to\mathbb{R}$ is defined by
 \begin{align*}
 \mathcal{J}(u,v)=\frac{1}{2}\|(-\Delta)^{1/4}u  \|_2^2+\frac{1}{2}\|(-\Delta)^{1/4}v  \|_2^2-\frac{1}{2}\int_{\mathbb{R}}(I_\mu*F(u,v))F(u,v)\mathrm{d}x.
 \end{align*}
  By using assumptions $(F_1)$ and $(F_2)$, it follows that for any $\zeta>0$, $q>1$ and $\alpha>\pi$, there exists $C>0$ such that
\begin{align*}
|F_{z_1}(z)|,|F_{z_2}(z)|\leq \zeta |z|^{\kappa}+C|z|^{q-1}(e^{\alpha |z|^2}-1), \quad \mbox{for any $z=(z_1,z_2)\in\mathbb{R}\times \mathbb{R}$},
\end{align*}
and using $(F_3)$, we have
\begin{align}\label{Ft}
|F(z)|\leq \zeta |z|^{\kappa+1}+C|z|^{q} (e^{\alpha |z|^2}-1), \quad \mbox{for any $z\in\mathbb{R}\times \mathbb{R}$}.
\end{align}
By \eqref{HLS} and \eqref{Ft}, we know $\mathcal{J}$ is well defined in $\mathcal{X}$ and $\mathcal{J}\in C^1(\mathcal{X},\mathbb{R})$ with
\begin{align*}
&\langle \mathcal{J}'(u,v),(\varphi,\psi)\rangle\\
=&
\frac{1}{2 \pi}\int_{\mathbb{R}}\int_{\mathbb{R}}\frac{[u(x)-u(y)][\varphi(x)-\varphi(y)]}{|x-y|^{2}}\mathrm{d}x\mathrm{d}y+\frac{1}{2 \pi}\int_{\mathbb{R}}\int_{\mathbb{R}}\frac{[v(x)-v(y)][\psi(x)-\psi(y)]}{|x-y|^{2}}\mathrm{d}x\mathrm{d}y\\
&-\int_{\mathbb{R}}(I_\mu*F(u,v))F_u(u,v)\varphi \mathrm{d}x-\int_{\mathbb{R}}(I_\mu*F(u,v))F_v(u,v)\phi \mathrm{d}x,
\end{align*}
for any $(u,v),(\varphi,\psi)\in \mathcal{X}$. Hence, a critical point of $\mathcal{J}$ on $\mathcal{S}$
corresponds to a solution of problem \eqref{problem}-\eqref{problem'}.

To understand the geometry of $\mathcal{J}|_{\mathcal{S}}$, for any $\beta\in \mathbb{R}$ and $u\in H^{1/2}(\mathbb{R})$, we define
\begin{equation*}
  \mathcal{H}(u,\beta)(x):=e^{\frac{\beta}{2}}u(e^\beta x),\quad \text{for a.e. $x\in \mathbb{R}$}.
\end{equation*}
One can easily check that $\| \mathcal{H}(u,\beta)\|_2=\|u\|_2$ for any $\beta\in \mathbb{R}$. As a consequence, for any $(u,v)\in \mathcal{S}$, it holds that $\mathcal{H}((u,v),\beta):=(\mathcal{H}(u,\beta),\mathcal{H}(v,\beta))\in \mathcal{S}$ for any $\beta\in \mathbb{R}$, and
$\mathcal{H}((u,v),\beta_1+\beta_2)=\mathcal{H}(\mathcal{H}((u,v),\beta_1),\beta_2)=\mathcal{H}(\mathcal{H}((u,v),\beta_2),\beta_1)$ for any $\beta_1,\beta_2\in \mathbb{R}$.
By Lemma \ref{mountain}, we find that $\mathcal{J}$ is unbounded from below on $\mathcal{S}$.
 It is well known that all critical points of $\mathcal{J}|_{\mathcal{S}}$ belong to the Poho\u{z}aev manifold (see \cite{chenwang2,dSS,MS2})
   \begin{align*}
   \mathcal{P}(a,b)=\big\{(u,v)\in \mathcal{S}: P(u,v)=0\big\},
   \end{align*}
where
  \begin{align*}
  P(u,v)=\|(-\Delta)^{1/4}u  \|_2^2+\|(-\Delta)^{1/4}v\|_2^2 -\int_{\mathbb{R}}(I_\mu*F(u,v))\widetilde{F}(u,v)\mathrm{d}x,
  \end{align*}
  where $\widetilde F(z)=z\cdot \nabla F(z)-(2-\mu)F(z)$.
This enlightens us to consider the minimization of $\mathcal{J}$ on $\mathcal{P}(a,b)$, i.e.,
\begin{align*}
   m(a,b)=\inf_{(u,v)\in\mathcal{P}(a,b)}\mathcal{J}(u,v).
   \end{align*}
Our task is to show that $m(a,b)$ is a critical level of $\mathcal{J}|_{\mathcal{S}}$. As will be shown in Lemma \ref{pa}, $\mathcal{P}(a,b)$ is nonempty, thus any critical point $(u,v)$ of $\mathcal{J}|_{\mathcal{S}}$ with $\mathcal{J}(u,v)=m(a,b)$ is a ground state solution of problem \eqref{problem}-\eqref{problem'}.

With a similar argument of \cite[Lemma 3.5]{BS2}, we have the following lemma.
\begin{lemma}\label{m}
Assume that $u_n\to u$ in $H^{1/2}(\mathbb{R})$ and $\beta_n\to \beta$ in $\mathbb{R}$ as $n\to\infty$, then $\mathcal{H}(u_n,\beta_n)\to \mathcal{H}(u,\beta)$ in $H^{1/2}(\mathbb{R})$ as $n\to\infty$.
\end{lemma}

\begin{lemma}   \label{f}
Assume that $(F_1)$-$(F_3)$ hold, let $\{(u_n,v_n)\}\subset \mathcal{S}$ be a bounded $(PS)_{m_{a,b}}$ sequence of $\mathcal{J}|_{\mathcal{S}}$, up to a subsequence,
  if $(u_n,v_n)\rightharpoonup (u,v)$ in $\mathcal{X}$ and
  \begin{equation*}
  \int_{\mathbb{R}}(I_\mu*F(u_n,v_n))\big[(u_n,v_n)\cdot \nabla F(u_n,v_n)\big]\mathrm{d}x\leq K_0
 \end{equation*}
 for some $K_0>0$, then for any $\phi\in C_0^\infty(\mathbb{R})$, we have
 \begin{align*}
 \int_{\mathbb{R}}(I_\mu*F(u_n,v_n))F_{u_n}(u_n,v_n)\phi\mathrm{d}x\to \int_{\mathbb{R}}(I_\mu*F(u,v))F_u(u,v)\phi\mathrm{d}x,\quad \mbox{as $n\rightarrow\infty$},
 \end{align*}
  \begin{align*}
 \int_{\mathbb{R}}(I_\mu*F(u_n,v_n))F_{v_n}(u_n,v_n)\phi\mathrm{d}x\to \int_{\mathbb{R}}(I_\mu*F(u,v))F_v(u,v)\phi\mathrm{d}x,\quad \mbox{as $n\rightarrow\infty$}.
 \end{align*}
 and
  \begin{align*}
 \int_{\mathbb{R}}(I_\mu*F(u_n,v_n))F(u_n,v_n)\mathrm{d}x\to \int_{\mathbb{R}}(I_\mu*F(u,v))F(u,v)\mathrm{d}x,\quad \mbox{as $n\rightarrow\infty$}.
 \end{align*}
\end{lemma}

 \begin{proof}
 The proof is similar to the one of \cite[Lemma 5.1]{chenwang2}, so we omit it.
 \end{proof}

\section{{\bfseries The estimation for the upper bound of $m(a,b)$}}\label{minimax}
In this section, by using the condition $(F_7)$, we estimate the upper bound of $m(a,b)$.

\begin{lemma}\label{mountain}
Assume that $(F_1)$-$(F_3)$ hold. Let $(u,v)\in \mathcal{S}$ be arbitrary but fixed, then we have

(i)  $\mathcal{J}(\mathcal{H}((u,v),\beta))\to0^+$ as $\beta\to -\infty$;

(ii)  $\mathcal{J}(\mathcal{H}((u,v),\beta))\to -\infty$ as $\beta\to +\infty$.
\end{lemma}

\begin{proof}
$(i)$ By a straightforward calculation, we have
\begin{align*}
\int_{\mathbb{R}}|\mathcal{H}(u,\beta) |^2 \mathrm{d}x=a^2,\ \ \int_{\mathbb{R}}|\mathcal{H}(v,\beta) |^2 \mathrm{d}x=b^2,
\end{align*}
\begin{align*}
\int_{\mathbb{R}}\int_{\mathbb{R}}\frac{|\mathcal{H}(u,\beta)(x)-\mathcal{H}(u,\beta)(y)|^2}{|x-y|^2}
\mathrm{d}x\mathrm{d}y
=e^{\beta}\int_{\mathbb{R}}\int_{\mathbb{R}}\frac{| u(x)-u(y)|^2}{|x-y|^2}
\mathrm{d}x\mathrm{d}y,
\end{align*}
\begin{equation*}
  \int_{\mathbb{R}}\int_{\mathbb{R}}\frac{|\mathcal{H}(v,\beta)(x)-\mathcal{H}(v,\beta)(y)|^2}{|x-y|^2}
\mathrm{d}x\mathrm{d}y
=e^{\beta}\int_{\mathbb{R}}\int_{\mathbb{R}}\frac{| v(x)-v(y)|^2}{|x-y|^2}
\mathrm{d}x\mathrm{d}y,
\end{equation*}
and
\begin{equation*}
   \int_{\mathbb{R}}|\mathcal{H}(u,\beta) |^\xi \mathrm{d}x=e^{\frac{(\xi-2)\beta}{2}}\int_{\mathbb{R}}|u|^\xi\mathrm{d}x, \ \ \text{for any $\xi>2$.}
\end{equation*}
Thus there exist $\beta_1<<0$  such that $\frac{2}{2-\mu}\big(\|(-\Delta)^{1/4}\mathcal{H}(u,\beta)  \|_{2}^2+\|(-\Delta)^{1/4}\mathcal{H}(v,\beta)  \|_{2}^2\big)<  1
$ for any $\beta<\beta_1$. Then using Lemma \ref{alge} with $k=2$, we know, for any $a_1+a_2<1$, there exists $p_1,p_2>1$ satisfying $\frac{1}{p_1}+\frac{1}{p_2}=1$ such that $p_1a_1<1$, $p_2a_2<1$. Hence, there exists $p_1,p_2>1$ satisfying $\frac{1}{p_1}+\frac{1}{p_2}=1$ such that
\begin{equation*}
\frac{2p_1}{2-\mu} \|(-\Delta)^{1/4}\mathcal{H}(u,\beta)  \|_{2}^2< 1,\quad \frac{2p_2}{2-\mu} \|(-\Delta)^{1/4}\mathcal{H}(v,\beta)  \|_{2}^2< 1,\quad \text{for any $\beta<\beta_1$}.
\end{equation*}
Fix $\alpha>\pi$ close to $\pi$ and $\nu>1$ close to $1$ such that
\begin{equation*}
  \frac{2p_1\alpha \nu}{2-\mu}\|(-\Delta)^{1/4}\mathcal{H}(u,\beta)  \|_{2}^2<\pi,\quad  \frac{2p_2\alpha \nu}{2-\mu}\|(-\Delta)^{1/4}\mathcal{H}(v,\beta)  \|_{2}^2<\pi,\quad \text{for any $\beta<\beta_1$}.
\end{equation*}
Then, for $\frac{1}{\nu}+\frac{1}{\nu'}=1$, using \eqref{HLS}, \eqref{att}, $(\ref{Ft})$, the H\"{o}lder and Young's inequality,
 we have
\begin{align}\label{tain1}
&\int_{\mathbb{R}}\big(I_\mu*F(\mathcal{H}((u,v),\beta) )\big)F(\mathcal{H}((u,v),\beta) )\mathrm{d}x \leq \|F(\mathcal{H}((u,v),\beta))\|^2_{\frac{2}{2-\mu}} 
\nonumber \\ \leq&
\zeta \|\mathcal{H}((u,v),\beta) \|_{\frac{2(\kappa+1)}{2-\mu}}^{2(\kappa+1)}+C \Big[\int_{\mathbb{R}} \big[(e^{\alpha |\mathcal{H}((u,v),\beta)|^2}-1) |\mathcal{H}((u,v),\beta)|^{q}\big]^{\frac{2}{2-\mu}} \mathrm{d}x\Big]^{2-\mu}\nonumber\\
\leq&\zeta \|\mathcal{H}((u,v),\beta) \|_{\frac{2(\kappa+1)}{2-\mu}}^{2(\kappa+1)}+C
\Big[\int_{\mathbb{R}}(e^{\frac{2 \alpha \nu}{2-\mu}|\mathcal{H}((u,v),\beta)|^2}-1)\mathrm{d}x\Big]^{\frac{2-\mu}{\nu}}\| \mathcal{H}((u,v),\beta) \|_{\frac{2q\nu'}{2-\mu}}^{2q}\nonumber\\
\leq &\zeta \|\mathcal{H}((u,v),\beta) \|_{\frac{2(\kappa+1)}{2-\mu}}^{2(\kappa+1)}\nonumber \\
&+C \Big[\frac{1}{p_1}\int_{\mathbb{R}}(e^{\frac{2 p_1\alpha \nu}{2-\mu}|\mathcal{H}(u,\beta)|^2}-1)\mathrm{d}x+\frac{1}{p_2}\int_{\mathbb{R}}(e^{\frac{2 p_2\alpha \nu}{2-\mu}|\mathcal{H}(v,\beta)|^2}-1)\mathrm{d}x\Big]^{\frac{2-\mu}{\nu}}\| \mathcal{H}((u,v),\beta) \|_{\frac{2q\nu'}{2-\mu}}^{2q}\nonumber\\
\leq & C \|\mathcal{H}((u,v),\beta) \|_{\frac{2(\kappa+1)}{2-\mu}}^{2(\kappa+1)}+ C\| \mathcal{H}((u,v),\beta) \|_{\frac{2q\nu'}{2-\mu}}^{2q}\nonumber\\
=& Ce^{(\kappa+\mu-1)\beta}\Big(\|u\|_{\frac{2(\kappa+1)}{2-\mu}}^{\frac{2(\kappa+1)}{2-\mu}}+\|v\|_{\frac{2(\kappa+1)}{2-\mu}}^{\frac{2(\kappa+1)}{2-\mu}}\Big)^{2-\mu}+
Ce^{\frac{(q\nu'+\mu-2)\beta}{\nu'}}\Big(\|u\|_{\frac{2q\nu'}{2-\mu}}^{\frac{2q\nu'}{2-\mu}}+\|v\|_{\frac{2q\nu'}{2-\mu}}^{\frac{2q\nu'}{2-\mu}}\Big)^{\frac{2-\mu}{\nu'}},
\end{align}
for any $\beta<\beta_1$.
Since $\kappa>2-\mu$, $q>1$ and $\nu'$ large enough, it follows that
\begin{align*}
 \mathcal{J}(\mathcal{H}((u,v),\beta))\geq& \frac{1}{2}e^\beta\big( \|(-\Delta)^{1/4}u\|_2^2+ \|(-\Delta)^{1/4}v\|_2^2\big)- Ce^{(\kappa+\mu-1)\beta}\Big(\|u\|_{\frac{2(\kappa+1)}{2-\mu}}^{\frac{2(\kappa+1)}{2-\mu}}+\|v\|_{\frac{2(\kappa+1)}{2-\mu}}^{\frac{2(\kappa+1)}{2-\mu}}\Big)^{2-\mu}\\
 &-
Ce^{\frac{(q\nu'+\mu-2)\beta}{\nu'}}\Big(\|u\|_{\frac{2q\nu'}{2-\mu}}^{\frac{2q\nu'}{2-\mu}}+\|v\|_{\frac{2q\nu'}{2-\mu}}^{\frac{2q\nu'}{2-\mu}}\Big)^{\frac{2-\mu}{\nu'}} \to0^+,\ \ \mbox{as} \ \beta\to-\infty.
\end{align*}

$(ii)$
For any fixed $\beta>>0$,  set
\begin{align*}
\mathcal{W}(t):=\frac{1}{2}\int_{\mathbb{R}}(I_\mu*F(tu,tv))F(tu,tv)\mathrm{d}x,\quad \text{for any $t>0$}.
\end{align*}
Using $(F_3)$, one has
\begin{align*}
\frac{\frac{\mathrm{d}\mathcal{W}(t)}{\mathrm{d}t}}{\mathcal{W}(t)}>\frac{2\theta}{t},\quad \text{for any $t>0$}.
\end{align*}
Thus, integrating this over $[1,e^{\frac{\beta}{2}}]$, we get
\begin{align}\label{fff}
\int_{\mathbb{R}}(I_\mu*F(e^{\frac{\beta}{2}}u,e^{\frac{\beta}{2}}v))F(e^{\frac{\beta}{2}}u,e^{\frac{\beta}{2}}v)\mathrm{d}x\geq e^{\theta\beta}\int_{\mathbb{R}}(I_\mu*F(u,v))F(u,v)\mathrm{d}x.
\end{align}
Hence,
\begin{align*}
\mathcal{J}(\mathcal{H}((u,v),\beta))\leq \frac{1}{2}e^\beta\big( \|(-\Delta)^{1/4}u\|_2^2+ \|(-\Delta)^{1/4}v\|_2^2\big)- \frac{1}{2} e^{(\theta+\mu-2)\beta}\int_{\mathbb{R}}(I_\mu*F(u,v))F(u,v)\mathrm{d}x.
\end{align*}
Since $\theta>3-\mu$, the above inequality yields that $\mathcal{J}(\mathcal{H}((u,v),\beta))\to -\infty$ as $\beta\to+\infty$.
\end{proof}

\begin{lemma}\label{pa}
Assume that $(F_1)$-$(F_3)$ and $(F_6)$ hold. Then for any fixed $(u,v)\in \mathcal{S}$, the function $\mathcal{J}(\mathcal{H}((u,v),\beta))$ reaches its unique maximum with positive level at a unique point $\beta_{(u,v)}\in\mathbb{R}$ such that $\mathcal{H}((u,v),\beta_{(u,v)})\in\mathcal{P}(a,b)$. Moreover, the mapping $(u,v)\rightarrow \beta_{(u,v)}$ is continuous in $(u,v)\in \mathcal{S}$.
\end{lemma}

\begin{proof}
From Lemma $\ref{mountain}$, there exists $\beta_{(u,v)}\in\mathbb{R}$ such that $$P(\mathcal{H}((u,v),\beta_{(u,v)}))=\frac{1}{2}\frac{\mathrm{d}}{\mathrm{d}\beta}\mathcal{J}(\mathcal{H}((u,v),\beta))|_{\beta=\beta_{(u,v)}}=0$$ and $\mathcal{J}(\mathcal{H}((u,v),\beta_{(u,v)}))>0$. Next, we prove the uniqueness of $\beta_{(u,v)}$.
For $(u,v)\in \mathcal{S}$ and $\beta\in\mathbb{R}$, we know
\begin{align*}
\mathcal{J}(\mathcal{H}((u,v),\beta))=&\frac{1}{2}e^\beta\big( \|(-\Delta)^{1/4}u\|_2^2+ \|(-\Delta)^{1/4}v\|_2^2\big)\\
&-\frac{1}{2} e^{(\mu-2)\beta} \int_{\mathbb{R}}  (I_\mu*F(e^{\frac{\beta}{2}}u,e^{\frac{\beta}{2}}v))F(e^{\frac{\beta}{2}} u,e^{\frac{\beta}{2}}v)\mathrm{d}x.
\end{align*}
Then taking into account that $P(\mathcal{H}((u,v),\beta_{(u,v)}))=0$, using $(F_6)$, we have
\begin{align*}
&\frac{\mathrm{d}^2}{\mathrm{d}\beta^2}\mathcal{J}(\mathcal{H}((u,v),\beta))|_{\beta=\beta_{(u,v)}}
\\=&\frac{1}{2}e^\beta\big( \|(-\Delta)^{1/4}u\|_2^2+ \|(-\Delta)^{1/4}v\|_2^2\big)\\
&-\frac{(2-\mu)^2}{2}e^{(\mu-2){\beta_{(u,v)}}}\int_{\mathbb{R}}(I_\mu*F(e^{\frac{{\beta_{(u,v)}}}{2}}u,e^{\frac{{\beta_{(u,v)}}}{2}}v))F(e^{\frac{{\beta_{(u,v)}}}{2}}u,e^{\frac{{\beta_{(u,v)}}}{2}}v)\mathrm{d}x
\\&+(\frac{7}{4}-\mu)e^{{\beta_{(u,v)}}}\int_{\mathbb{R}}(I_\mu*F(e^{\frac{{\beta_{(u,v)}}}{2}}u,e^{\frac{{\beta_{(u,v)}}}{2}}v))\\&\times\bigg[\big(e^{\frac{{\beta_{(u,v)}}}{2}}u,e^{\frac{{\beta_{(u,v)}}}{2}}v\big)\cdot \Big(\frac{\partial F(e^{\frac{{\beta_{(u,v)}}}{2}}u,e^{\frac{{\beta_{(u,v)}}}{2}}v) }{\partial(e^{\frac{{\beta_{(u,v)}}}{2}}u)},\frac{\partial F(e^{\frac{{\beta_{(u,v)}}}{2}}u,e^{\frac{{\beta_{(u,v)}}}{2}}v) }{\partial(e^{\frac{{\beta_{(u,v)}}}{2}}v)} \Big)\bigg] \mathrm{d}x\\
&-\frac{1}{4}e^{(\mu-2){\beta_{(u,v)}}}\int_{\mathbb{R}}\bigg(I_\mu* \bigg[\big(e^{\frac{{\beta_{(u,v)}}}{2}}u,e^{\frac{{\beta_{(u,v)}}}{2}}v\big)\cdot \Big(\frac{\partial F(e^{\frac{{\beta_{(u,v)}}}{2}}u,e^{\frac{{\beta_{(u,v)}}}{2}}v) }{\partial(e^{\frac{{\beta_{(u,v)}}}{2}}u)},\frac{\partial F(e^{\frac{{\beta_{(u,v)}}}{2}}u,e^{\frac{{\beta_{(u,v)}}}{2}}v) }{\partial(e^{\frac{{\beta_{(u,v)}}}{2}}v)} \Big)\bigg]\bigg)\\&\times\bigg[\big(e^{\frac{{\beta_{(u,v)}}}{2}}u,e^{\frac{{\beta_{(u,v)}}}{2}}v\big )\cdot \Big(\frac{\partial F(e^{\frac{{\beta_{(u,v)}}}{2}}u,e^{\frac{{\beta_{(u,v)}}}{2}}v) }{\partial(e^{\frac{{\beta_{(u,v)}}}{2}}u)},\frac{\partial F(e^{\frac{{\beta_{(u,v)}}}{2}}u,e^{\frac{{\beta_{(u,v)}}}{2}}v) }{\partial(e^{\frac{{\beta_{(u,v)}}}{2}}v)} \Big)\bigg]\mathrm{d}x\\&
-\frac{1}{4}e^{(\mu-2){\beta_{(u,v)}}}\int_{\mathbb{R}}(I_\mu*F(e^{\frac{{\beta_{(u,v)}}}{2}}u,e^{\frac{{\beta_{(u,v)}}}{2}}v))\bigg[(e^{\frac{{\beta_{(u,v)}}}{2}}u)^2\frac{\partial^2 F(e^{\frac{{\beta_{(u,v)}}}{2}}u,e^{\frac{{\beta_{(u,v)}}}{2}}v)}{\partial (e^{\frac{{\beta_{(u,v)}}}{2}}u)^2}\\&+(e^{\frac{{\beta_{(u,v)}}}{2}}v)^2\frac{\partial^2 F(e^{\frac{{\beta_{(u,v)}}}{2}}u,e^{\frac{{\beta_{(u,v)}}}{2}}v)}{\partial (e^{\frac{{\beta_{(u,v)}}}{2}}v)^2}\bigg]\mathrm{d}x\\
=&-\frac{(2-\mu)(3-\mu)}{2}e^{(\mu-2){\beta_{(u,v)}}}\int_{\mathbb{R}}(I_\mu*F(e^{\frac{{\beta_{(u,v)}}}{2}}u,e^{\frac{{\beta_{(u,v)}}}{2}}v))F(e^{\frac{{\beta_{(u,v)}}}{2}}u,e^{\frac{{\beta_{(u,v)}}}{2}}v)\mathrm{d}x
\\&+(\frac{9}{4}-\mu)e^{{\beta_{(u,v)}}}\int_{\mathbb{R}}(I_\mu*F(e^{\frac{{\beta_{(u,v)}}}{2}}u,e^{\frac{{\beta_{(u,v)}}}{2}}v))\\&\times\bigg[\big(e^{\frac{{\beta_{(u,v)}}}{2}}u,e^{\frac{{\beta_{(u,v)}}}{2}}v\big)\cdot \Big(\frac{\partial F(e^{\frac{{\beta_{(u,v)}}}{2}}u,e^{\frac{{\beta_{(u,v)}}}{2}}v) }{\partial(e^{\frac{{\beta_{(u,v)}}}{2}}u)},\frac{\partial F(e^{\frac{{\beta_{(u,v)}}}{2}}u,e^{\frac{{\beta_{(u,v)}}}{2}}v) }{\partial(e^{\frac{{\beta_{(u,v)}}}{2}}v)} \Big)\bigg] \mathrm{d}x\\&-\frac{1}{4}e^{(\mu-2){\beta_{(u,v)}}}\int_{\mathbb{R}}\bigg(I_\mu* \bigg[\big(e^{\frac{{\beta_{(u,v)}}}{2}}u,e^{\frac{{\beta_{(u,v)}}}{2}}v\big)\cdot \Big(\frac{\partial F(e^{\frac{{\beta_{(u,v)}}}{2}}u,e^{\frac{{\beta_{(u,v)}}}{2}}v) }{\partial(e^{\frac{{\beta_{(u,v)}}}{2}}u)},\frac{\partial F(e^{\frac{{\beta_{(u,v)}}}{2}}u,e^{\frac{{\beta_{(u,v)}}}{2}}v) }{\partial(e^{\frac{{\beta_{(u,v)}}}{2}}v)} \Big)\bigg]\bigg)\\&\times\bigg[\big(e^{\frac{{\beta_{(u,v)}}}{2}}u,e^{\frac{{\beta_{(u,v)}}}{2}}v\big)\cdot \Big(\frac{\partial F(e^{\frac{{\beta_{(u,v)}}}{2}}u,e^{\frac{{\beta_{(u,v)}}}{2}}v) }{\partial(e^{\frac{{\beta_{(u,v)}}}{2}}u)},\frac{\partial F(e^{\frac{{\beta_{(u,v)}}}{2}}u,e^{\frac{{\beta_{(u,v)}}}{2}}v) }{\partial(e^{\frac{{\beta_{(u,v)}}}{2}}v)} \Big)\bigg]\mathrm{d}x\\&
-\frac{1}{4}e^{(\mu-2){\beta_{(u,v)}}}\int_{\mathbb{R}}(I_\mu*F(e^{\frac{{\beta_{(u,v)}}}{2}}u,e^{\frac{{\beta_{(u,v)}}}{2}}v))\bigg[(e^{\frac{{\beta_{(u,v)}}}{2}}u)^2\frac{\partial^2 F(e^{\frac{{\beta_{(u,v)}}}{2}}u,e^{\frac{{\beta_{(u,v)}}}{2}}v)}{\partial (e^{\frac{{\beta_{(u,v)}}}{2}}u)^2}\\&+(e^{\frac{{\beta_{(u,v)}}}{2}}v)^2\frac{\partial^2 F(e^{\frac{{\beta_{(u,v)}}}{2}}u,e^{\frac{{\beta_{(u,v)}}}{2}}v)}{\partial (e^{\frac{{\beta_{(u,v)}}}{2}}v)^2}\bigg]\mathrm{d}x\\
=:&\frac{1}{4}\int_{\mathbb{R}}\int_{\mathbb{R}}\frac{G}{|x-y|^\mu}\mathrm{d}x\mathrm{d}y<0,
\end{align*}
this prove the uniqueness of ${\beta_{(u,v)}}_{(u,v)}$,
where
\begin{align*}
  G=&(3-\mu)F(\mathcal{H}((u(y),v(y)),\beta_{(u(y),v(y))}))\widetilde{F}(\mathcal{H}((u(x),v(x)),\beta_{(u(x),v(x))}))\\&
  -F(\mathcal{H}((u(y),v(y)),\beta_{(u(y),v(y))}))\times\bigg[\mathcal{H}((u(x),v(x)),\beta_{(u(x),v(x))}) \\& \cdot\Big(\frac{\partial  \widetilde{F}(\mathcal{H}((u(x),v(x)),\beta_{(u(x),v(x))}))}{\partial (\mathcal{H}(u(x),\beta_{(u(x),v(x))}))},\frac{\partial \widetilde{F}(\mathcal{H}((u(x),v(x)),\beta_{(u(x),v(x))}))}{\partial (\mathcal{H}(v(x),\beta_{(u(x),v(x))}))}\Big)\bigg]\\&
  -\widetilde{F}(\mathcal{H}((u(y),v(y)),\beta_{(u(y),v(y))}))\\&\times[\widetilde{F}(\mathcal{H}((u(x),v(x)),\beta_{(u(x),v(x))}))-F(\mathcal{H}((u(x),v(x)),\beta_{(u(x),v(x))}))].
\end{align*}

From the above arguments, we know the mapping $(u,v)\rightarrow \beta_{(u,v)}$ is well defined. Let $\{(u_n,v_n)\}\subset \mathcal{S}$ be a sequence such that $(u_n,v_n)\to (u,v)\neq (0,0)$ in $\mathcal{X}$ as $n\to \infty$. We only need to prove that, up to a subsequence, $\beta_{(u_n,v_n)}\to \beta_{(u,v)}$ in $\mathbb{R}$ as $n\to \infty$.

On the one hand, if up to a subsequence, $\beta_{(u_n,v_n)}\to+\infty$ as $n\to\infty$, then by \eqref{fff} and $ (u,v)\neq (0,0)$, we have
\begin{align*}
0&\leq\lim_{n\to\infty}e^{-\beta_{(u_n,v_n)}}\mathcal{J}(\mathcal{H}((u_n,v_n),\beta_{(u_n,v_n)}))
\\&\leq\lim_{n\to\infty} \frac{1}{2}\Big[ \|(-\Delta)^{1/4}u_n\|_{2}^2+\|(-\Delta)^{1/4}v_n\|_{2}^2
-e^{(\theta+\mu-3)\beta_{(u_n,v_n)}} \int_{\mathbb{R}}  (I_\mu*F(u_n,v_n))F(u_n,v_n)\mathrm{d}x\Big]=-\infty,
\end{align*}
which is a contradiction. Hence, $\{\beta_{(u_n,v_n)}\}$ is bounded from above.

On the other hand, by Lemma $\ref{m}$,
we know $\mathcal{H}((u_n,v_n),\beta_{(u,v)})\to \mathcal{H}((u,v),\beta_{(u,v)})$ in $\mathcal{X}$ as $n\to\infty$. Then
\begin{align*}
\mathcal{J}(\mathcal{H}((u_n,v_n),\beta_{(u_n,v_n)}))\geq \mathcal{J}(\mathcal{H}((u_n,v_n),\beta_{(u,v)}))=\mathcal{J}(\mathcal{H}((u,v),\beta_{(u,v)}))+o_n(1),
\end{align*}
and thus
\begin{align*}
\liminf_{n\to\infty}\mathcal{J}(\mathcal{H}((u_n,v_n),\beta_{(u_n,v_n)}))\geq \mathcal{J}(\mathcal{H}((u,v),\beta_{(u,v)}))>0.
\end{align*}
If up to a subsequence, $\beta_{(u_n,v_n)}\to -\infty$ as $n\to\infty$, using $(F_3)$, we get
\begin{align*}
\mathcal{J}(\mathcal{H}((u_n,v_n),\beta_{(u_n,v_n)}))\leq \frac{1}{2}e^{\beta_{(u_n,v_n)}}\big(\|(-\Delta)^{1/4}u_n\|_2^2+\|(-\Delta)^{1/4}v_n\|_2^2\big)\to0, \quad \text{as $n\rightarrow\infty$},
\end{align*}
which is impossible. So we get $\beta_{(u_n,v_n)}$ is bounded from below. Up to a subsequence, we assume that $\beta_{(u_n,v_n)}\to \beta_0$ as $n\to\infty$. Since $(u_n,v_n)\to (u,v)$ in $\mathcal{X}$, then $\mathcal{H}((u_n,v_n),\beta_{(u_n,v_n)})\to\mathcal{H}((u,v),\beta_0)$ in $\mathcal{X}$ as $n\to\infty$. Moreover, by $P(\mathcal{H}((u_n,v_n),\beta_{(u_n,v_n)}))=0$, it follows that $P(\mathcal{H}((u,v),\beta_0))=0$. By the uniqueness of $\beta_{(u,v)}$, we get $\beta_{(u,v)}=\beta_0$ and the conclusion follows.
\end{proof}

\begin{lemma}\label{ous11}
Assume that $(F_1)$-$(F_3)$ hold, then there exists $\gamma>0$ small enough such that
\begin{align*}
\mathcal{J}(u,v)\geq \frac{1}{4}\big(\|(-\Delta)^{1/4}u\|_2^2+\|(-\Delta)^{1/4}v\|_2^2\big)
\quad \mbox{and}\quad
P(u,v)\geq\frac{1}{2}\big(\|(-\Delta)^{1/4}u\|_2^2+\|(-\Delta)^{1/4}v\|_2^2\big)
\end{align*}
for $(u,v)\in \mathcal{S}$ satisfying $\|(-\Delta)^{1/4}u\|_2^2+\|(-\Delta)^{1/4}v\|_2^2\leq \gamma$.
\end{lemma}

\begin{proof}
If $\gamma<\frac{2-\mu}{2}$, then $\frac{2}{2-\mu}\big(\|(-\Delta)^{1/4}u\|_2^2+\|(-\Delta)^{1/4}v\|_2^2\big)< 1$.
From \eqref{gns} and $(\ref{tain1})$, we obtain
\begin{align*}
&\int_{\mathbb{R}}\big(I_\mu*F(u,v)\big)F(u,v)\mathrm{d}x \\ \leq& C\Big(\|u\|_{\frac{2(\kappa+1)}{2-\mu}}^{\frac{2(\kappa+1)}{2-\mu}}+\|v\|_{\frac{2(\kappa+1)}{2-\mu}}^{\frac{2(\kappa+1)}{2-\mu}}\Big)^{2-\mu}+
C\Big(\|u\|_{\frac{2q\nu'}{2-\mu}}^{\frac{2q\nu'}{2-\mu}}+\|v\|_{\frac{2q\nu'}{2-\mu}}^{\frac{2q\nu'}{2-\mu}}\Big)^{\frac{2-\mu}{\nu'}}\\ \leq&
C\Big(a^2 \|(-\Delta)^{1/4}u\|_2^{\frac{2(\kappa+\mu-1)}{2-\mu}}+b^2 \|(-\Delta)^{1/4}v\|_2^{\frac{2(\kappa+\mu-1)}{2-\mu}}\Big)^{2-\mu}\\
&+ C\Big(a^2\|(-\Delta)^{1/4}u\|_2^{\frac{2(q\nu'+\mu-2)}{2-\mu}}+b^2\|(-\Delta)^{1/4}v\|_2^{\frac{2(q\nu'+\mu-2)}{2-\mu}}\Big)^{\frac{2-\mu}{\nu'}}
\\ \leq& C\big( \|(-\Delta)^{1/4}u\|_2^{2(\kappa+\mu-1)}+ \|(-\Delta)^{1/4}v\|_2^{2(\kappa+\mu-1)}\big)\\
&+ C\big(\|(-\Delta)^{1/4}u\|_2^{\frac{2(q\nu'+\mu-2)}{\nu'}}+\|(-\Delta)^{1/4}v\|_2^{\frac{2(q\nu'+\mu-2)}{\nu'}}\big)\\
\leq &  C\big( \gamma^{\kappa+\mu-2}+ \gamma^{q-1+\frac{\mu-2}{\nu'}}\big)\|(-\Delta)^{1/4}u\|_2^2+C\big( \gamma^{\kappa+\mu-2}+ \gamma^{q-1+\frac{\mu-2}{\nu'}}\big)\|(-\Delta)^{1/4}v\|_2^2.
\end{align*}
Since $\kappa>2-\mu$, $q>1$ and $\nu'=\frac{\nu}{\nu-1}$ large enough, choosing $0<\gamma<\frac{2-\mu}{2}$ small enough, we conclude the result.
\end{proof}

\begin{lemma}\label{ous1}
Assume that $(F_1)$-$(F_3)$ and $(F_6)$ hold, then we have
\begin{equation*}
\inf\limits_{(u,v)\in\mathcal{P}(a,b)}\big(\|(-\Delta)^{1/4}u\|_2+\|(-\Delta)^{1/4}v\|_2\big)>0\ \ \text{and} \ \ m(a,b)>0.
\end{equation*}
\end{lemma}

\begin{proof}
By Lemma $\ref{pa}$, we know $\mathcal{P}(a,b)$ is nonempty. Supposed that there exists a sequence $\{(u_n,v_n)\}\subset \mathcal{P}(a,b)$ such that $\|(-\Delta)^{1/4}u_n\|_2+\|(-\Delta)^{1/4}v_n\|_2\to0$ as $n\to\infty$,
then by Lemma $\ref{ous11}$, up to subsequence,
\begin{align*}
0=P(u_n,v_n)\geq\frac{1}{2} \big(\|(-\Delta)^{1/4}u_n\|_2^2+\|(-\Delta)^{1/4}v_n\|_2^2\big)\geq0,
\end{align*}
which implies that $\|(-\Delta)^{1/4}u_n\|_2^2=\|(-\Delta)^{1/4}v_n\|_2^2=0$ for any $n\in\mathbb{N}^+$.
By $(F_3)$ and $P(u_n,v_n)=0$, we have
\begin{align*}
0=&(2-\mu)\int_{\mathbb{R}}(I_\mu*F(u_n,v_n))F(u_n,v_n)\mathrm{d}x-\int_{\mathbb{R}}(I_\mu*F(u_n,v_n))\big[(u_n,v_n)\cdot \nabla F(u_n,v_n)\big]\mathrm{d}x\\
\leq& \big(\frac{2-\mu}{\theta}-1\big) \int_{\mathbb{R}}(I_\mu*F(u_n,v_n))\big[(u_n,v_n)\cdot \nabla F(u_n,v_n)\big]\mathrm{d}x\leq0.
\end{align*}
So $u_n,v_n\to0$ a.e. in $\mathbb{R}$, which contradicts  $a,b>0$.

From Lemma $\ref{pa}$, we know that for any $(u,v)\in\mathcal{P}(a,b)$,
\begin{align*}
\mathcal{J}(u,v)=\mathcal{J}(\mathcal{H}((u,v),0))\geq \mathcal{J}(\mathcal{H}((u,v),\beta)),\quad \text{for any $\beta\in\mathbb{R}$} .
\end{align*}
 Let $\gamma>0$ be the number given by Lemma $\ref{ous11}$ and $e^{\beta}=\frac{\gamma}{ \|(-\Delta)^{1/4}u\|_2^2+\|(-\Delta)^{1/4}v\|_2^2}$, then
 $$ \|(-\Delta)^{1/4}\mathcal{H}(u,\beta)\|_2^2+\|(-\Delta)^{1/4}\mathcal{H}(v,\beta)\|_2^2=\gamma.$$ Applying Lemma $\ref{ous11}$ again, we deduce that
\begin{align*}
\mathcal{J}(u,v)\geq \mathcal{J}(\mathcal{H}((u,v),\beta))\geq \frac{1}{4}\big(\|(-\Delta)^{1/4}\mathcal{H}(u,\beta)\|_2^2+\|(-\Delta)^{1/4}\mathcal{H}(v,\beta)\|_2^2\big)
=\frac{\gamma}{4}>0.
\end{align*}
This completes the proof.
\end{proof}

In order to estimate the upper bound of $m(a,b)$, let us consider the following sequence of nonnegative functions (see \cite{Taka}) supported in $B_1(0)$ given by
  \begin{align*}
\varpi_n(x)=\frac{1}{\sqrt{\pi}}
	 \begin{cases}
  \sqrt{\log n},\ \  &\mbox{for $|x|<\frac{1}{n}$,}\\
 \frac{\log{\frac{1}{|x|}}}{\sqrt{\log n}},\ \  &\mbox{for $\frac{1}{n}\leq |x|\leq 1$,}\\
  0,\ \  &\mbox{for $|x|> 1$.}
      \end{cases}
 \end{align*}
One can check that $\varpi_n\in H^{1/2}(\mathbb{R})$. A direct calculation shows that
\begin{align*}
\|(-\Delta)^{1/4}\varpi_n\|_2^2=1+o(1),
\end{align*}
\begin{align*}
\delta_n:=\|\varpi_n\|_2^2=&\int_{-\frac{1}{n}}^{\frac{1}{n}}\frac{\log n}{\pi}\mathrm{d}x+
\int_{-1}^{-\frac{1}{n}}\frac{(\log|x|)^2}{\pi\log n}\mathrm{d}x+
\int_{\frac{1}{n}}^1\frac{(\log|x|)^2}{\pi\log n}\mathrm{d}x\\
=&\frac{4}{\pi}(\frac{1}{\log n}-\frac{1}{n\log n }-\frac{1}{n})=\frac{4}{\pi\log n}+o(\frac{1}{\log n}).
\end{align*}
For any $c>0$, let $\omega_n^c:=\frac{c\varpi_n}{\|\varpi_n\|_2}$, then $\omega_n^c \in S(c)$ and
\begin{align}\label{wnx}
\omega_n^c(x)=\frac{c}{2}
	 \begin{cases}
  \log n (1+o(1)),\ \  &\mbox{for $|x|<\frac{1}{n}$,}\\
 \log{\frac{1}{|x|}}(1+o(1)),\ \  &\mbox{for $\frac{1}{n}\leq |x|\leq 1$,} \\
  0,\ \  &\mbox{for $|x|> 1$.}
      \end{cases}
\end{align}
Furthermore, we have
\begin{align}\label{wn}
\|(-\Delta)^{1/4}\omega_n^c\|_2^2&=\int_{\mathbb{R}}\int_{\mathbb{R}}\frac{\big|\frac{c}{\sqrt{\delta_n}}\varpi_n(x)-\frac{c}{\sqrt{\delta_n}}\varpi_n(y)\big|^2}{|x-y|^2}\mathrm{d}x\mathrm{d}y\nonumber\\
&=\frac{c^2}{\delta_n}\|(-\Delta)^{1/4}\varpi_n\|_2^2=\frac{\pi c^2\log n}{4}(1+o(1)).
\end{align}
For any $t>0$, let
\begin{align*}
\Phi_n(t):=&\mathcal{J}(t\omega_n^a(t^{2}x),t\omega_n^b(t^{2}x))=\frac{t^2}{2}\big(\|(-\Delta)^{1/4}\omega_n^a\|_2^2+\|(-\Delta)^{1/4}\omega_n^b\|_2^2\big) \nonumber\\
&-\frac{1}{2}t^{2(\mu-2)}\int_{\mathbb{R}}(I_\mu*F(t\omega_n^a,t\omega_n^b))F(t\omega_n^a,t\omega_n^b)\mathrm{d}x.
\end{align*}
From Lemmas {\ref{pa}} and $\ref{ous1}$, we infer that $m(a,b)=\inf\limits_{(u,v)\in \mathcal{S}}\max\limits_{\beta\in\mathbb{R}}\mathcal{J}(\mathcal{H}((u,v),\beta))>0$, this together with $(\omega_n^a,\omega_n^b)\in \mathcal{S}$ yields that
\begin{align*}
m(a,b)\leq \max_{\beta\in\mathbb{R}}\mathcal{J}(\mathcal{H}((\omega_n^a,\omega_n^b),\beta))=\max_{t>0}\Phi_n(t).
\end{align*}

\begin{lemma}\label{attain}
Assume that $(F_1)$-$(F_3)$ hold, then for any fixed $n\in \mathbb{N}^+$, $\max\limits_{t\geq0}\Phi_n(t)>0$ is attained at some $t_n>0$.
\end{lemma}
\begin{proof}
For any fixed $p_1,p_2>1$ satisfying $\frac{1}{p_1}+\frac{1}{p_2}=1$, as $t>0$ small enough, one can fix $\alpha>\pi$ close to $\pi$ and $\nu>1$ close to $1$ such that
\begin{equation*}
  \frac{2p_1\alpha \nu}{2-\mu}\|(-\Delta)^{1/4}(t\omega_n^a)\|_2^2<\pi,\quad  \frac{2p_2\alpha \nu}{2-\mu}\|(-\Delta)^{1/4}(t\omega_n^b)\|_2^2<\pi.
\end{equation*}
Arguing as $(\ref{tain1})$, by \eqref{att}, for $\nu'=\frac{\nu}{\nu-1}$, we have
\begin{align*}
&t^{2(\mu-2)}\int_{\mathbb{R}}(I_\mu*F(t\omega_n^a,t\omega_n^b))F(t\omega_n^a,t\omega_n^b)\mathrm{d}x \\ \leq& Ct^{2(\mu-2)}\Big(\|t\omega_n^a\|_{\frac{2(\kappa+1)}{2-\mu}}^{\frac{2(\kappa+1)}{2-\mu}}+\|t\omega_n^b\|_{\frac{2(\kappa+1)}{2-\mu}}^{\frac{2(\kappa+1)}{2-\mu}}\Big)^{2-\mu}+
Ct^{2(\mu-2)}t^{\frac{2(2-\mu)}{\nu}}\Big(\|t\omega_n^a\|_{\frac{2q\nu'}{2-\mu}}^{\frac{2q\nu'}{2-\mu}}+\|t\omega_n^b\|_{\frac{2q\nu'}{2-\mu}}^{\frac{2q\nu'}{2-\mu}}\Big)^{\frac{2-\mu}{\nu'}}
\\=&Ct^{2(\kappa+\mu-1)}\Big(\|\omega_n^a\|_{\frac{2(\kappa+1)}{2-\mu}}^{\frac{2(\kappa+1)}{2-\mu}}+\|\omega_n^b\|_{\frac{2(\kappa+1)}{2-\mu}}^{\frac{2(\kappa+1)}{2-\mu}}\Big)^{2-\mu}+Ct^{2(q+\mu-2)+\frac{2(2-\mu)}{\nu}}
\Big(\|\omega_n^a\|_{\frac{2q\nu'}{2-\mu}}^{\frac{2q\nu'}{2-\mu}}+\|\omega_n^b\|_{\frac{2q\nu'}{2-\mu}}^{\frac{2q\nu'}{2-\mu}}\Big)^{\frac{2-\mu}{\nu'}}.
\end{align*}
Since $\kappa>2-\mu$, $q>1$, and $\nu$ close to $1$, we have $\Phi_n(t)>0$ for $t>0$ small enough. For $t>0$ large, by \eqref{fff}, we obtain
\begin{align*}
t^{2(\mu-2)}\int_{\mathbb{R}}(I_\mu*F(t\omega_n^a,t\omega_n^b))F(t\omega_n^a,t\omega_n^b)\mathrm{d}x \geq
t^{2(\theta+\mu-2)}\int_{\mathbb{R}}(I_\mu*F(\omega_n^a,\omega_n^b))F(\omega_n^a,\omega_n^b)\mathrm{d}x.
\end{align*}
Since $\theta>3-\mu$, we obtain $\Phi_n(t)<0$ for $t>0$ large enough.
Thus $\max\limits_{t\geq0}\Phi_n(t)>0$ is attained at some $t_n>0$.
\end{proof}

\begin{lemma}\label{contr}
Assume that $(F_1)$-$(F_3)$ and $(F_7)$ hold, then there exists $n\in \mathbb{N}^+$ large such that
\begin{align*}
\max_{t\geq0}\Phi_n(t)<\frac{2-\mu}{4}.
\end{align*}
\end{lemma}

\begin{proof}
First, we have the following estimation
\begin{align*}
  \int_{-\frac{1}{n}}^{\frac{1}{n}}\int_{-\frac{1}{n}}^{\frac{1}{n}}\frac{\mathrm{d}x\mathrm{d}y}{|x-y|^\mu}=\frac{2^{3-\mu}}{(1-\mu)(2-\mu)}(\frac{1}{n})^{2-\mu}=:C_\mu (\frac{1}{n})^{2-\mu}.
\end{align*}
By Lemma $\ref{attain}$, we know $\max\limits_{t\geq0}\Phi_n(t)$ is attained at some $t_n>0$. So $t_n$ satisfies
\begin{align*}
\frac{d}{dt}\Phi_n(t)\Big|_{t=t_n}=0.
\end{align*}
By $(F_3)$, we have
\begin{align}\label{twn}
&t_n^2\big(\|(-\Delta)^{1/4}\omega_n^a\|_2^2+\|(-\Delta)^{1/4}\omega_n^b\|_2^2\big)\nonumber\\=&(\mu-2)t_n^{2(\mu-2)}\int_{\mathbb{R}}(I_\mu*F(t_n\omega_n^a,t_n\omega_n^b))F(t_n\omega_n^a,t_n\omega_n^b)\mathrm{d}x\nonumber\\
&+t_n^{2(\mu-2)}\int_{\mathbb{R}}(I_\mu*F(t_n\omega_n^a,t_n\omega_n^b))\Big(\frac{\partial F(t_n\omega_n^a,t_n\omega_n^b)}{\partial (t_n\omega_n^a)}t_n\omega_n^a+\frac{\partial F(t_n\omega_n^a,t_n\omega_n^b)}{\partial (t_n\omega_n^b)}t_n\omega_n^b\Big)\mathrm{d}x\nonumber\\
 \geq& \frac{\theta+\mu-2}{\theta}t_n^{2(\mu-2)}\int_{\mathbb{R}}(I_\mu*F(t_n\omega_n^a,t_n\omega_n^b))\Big(\frac{\partial F(t_n\omega_n^a,t_n\omega_n^b)}{\partial (t_n\omega_n^a)}t_n\omega_n^a+\frac{\partial F(t_n\omega_n^a,t_n\omega_n^b)}{\partial (t_n\omega_n^b)}t_n\omega_n^b\Big)\mathrm{d}x.
\end{align}

By $(F_7)$, for any $\varepsilon>0$. there exists $R_\varepsilon>0$ such that for any $|z_1|,|z_2|\geq R_\varepsilon$,
\begin{equation}\label{ftF}
 F(z)[z\cdot \nabla F(z)]\geq (\beta_0-\varepsilon)e^{2\pi |z|^2}.
\end{equation}

$\bf{Case\ 1}$.
 If $\lim\limits_{n\to\infty}t_n^2\log n=0$, then $\lim\limits_{n\to\infty}t_n=0$. By $(\ref{wn})$, we have
 \begin{equation*}
 \frac{t_n^2}{2}\big(\|(-\Delta)^{1/4}\omega_n^a\|_2^2+\|(-\Delta)^{1/4}\omega_n^b\|_2^2\big)\to0,\quad \text{as $n\rightarrow\infty$}.
\end{equation*}
Note that $F(t_n\omega_n^a,t_n\omega_n^b)>0$ by $(F_3)$, so we have
\begin{align*}
\Phi_n(t_n)\leq\frac{t_n^2}{2}\big(\|(-\Delta)^{1/4}\omega_n^a\|_2^2+\|(-\Delta)^{1/4}\omega_n^b\|_2^2\big),
\end{align*}
which implies that $\lim\limits_{n\to+\infty}\Phi_n(t_n)=0$, and we conclude.

$\bf{Case\ 2}.
$ If $\lim\limits_{n\to\infty}t_n^2\log n=l\in(0,+\infty]$. From $(\ref{wnx})$-$(\ref{ftF})$, we have
\begin{align*}
&t_n^2\Big(\frac{\pi (a^2+b^2) \log n}{4}(1+o(1))\Big)\\
\geq&\frac{\theta+\mu-2}{\theta}t_n^{2(\mu-2)}\int_{B_{\frac{1}{n}}(0)}\int_{B_{\frac{1}{n}}(0)}\frac{K}{|x-y|^\mu}\mathrm{d}x\mathrm{d}y\\
\geq&\frac{(\theta+\mu-2)(\beta_0-\varepsilon)^2}{\theta}t_n^{2(\mu-2)}e^{\frac{\pi (a^2+b^2)t_n^2\log ^2n(1+o(1))}{2}}\int_{-\frac{1}{n}}^{\frac{1}{n}}\int_{-\frac{1}{n}}^{\frac{1}{n}}\frac{\mathrm{d}x\mathrm{d}y}{|x-y|^\mu}\\
=&\frac{C_\mu(\theta+\mu-2)(\beta_0-\varepsilon)^2}{\theta} t_n^{2(\mu-2)}e^{\big(\frac{\pi (a^2+b^2)t_n^2\log n (1+o(1)))}{2}-(2-\mu)\big)\log n},
\end{align*}
where
\begin{equation*}
  K:=F(t_n\omega_n^a(y),t_n\omega_n^b(y)) \Big(\frac{\partial F(t_n\omega_n^a(x),t_n\omega_n^b(x))}{\partial (t_n\omega_n^a(x))}t_n\omega_n^a(x)+\frac{\partial F(t_n\omega_n^a(x),t_n\omega_n^b(x))}{\partial (t_n\omega_n^b(x))}t_n\omega_n^b(x)\Big).
\end{equation*}

(i) If $l=+\infty$, we get a contradiction from the inequality above. So $l\in(0,+\infty)$ and $\lim\limits_{n\rightarrow\infty}t_n=0$. In particular, using the inequality above again and letting $n\to +\infty$, we have $l\in\big(0,\frac{2(2-\mu)}{\pi (a^2+b^2)}\big]$.

(ii) If $l\in\big(0,\frac{2(2-\mu)}{\pi (a^2+b^2)}\big)$, then by $(\ref{wn})$, we get
\begin{align*}
\lim\limits_{n\to\infty}\Phi_n(t_n)\leq \frac{1}{2}\lim\limits_{n\to\infty}t_n^2\big(\|(-\Delta)^{1/4}\omega_n^a\|_2^2+\|(-\Delta)^{1/4}\omega_n^b\|_2^2\big)=\frac{\pi (a^2+b^2) l}{8}<\frac{2-\mu}{4}.
\end{align*}

(iii) If $l=\frac{2(2-\mu)}{\pi (a^2+b^2)}$, by the definition of $\omega_n^a$ and $\omega_n^b$, we can find that
\begin{equation*}
  Q_n:=\frac{\pi (a^2+b^2)t_n^2\log n (1+o(1))}{2}-(2-\mu)\rightarrow 0^+,\quad \text{as $n\rightarrow\infty$}.
\end{equation*}
Using the Taylor's formula, we have
\begin{equation*}
  n^{Q_n}=1+Q_n\log n+\frac{Q_n^2\log ^2 n}{2}+\cdots\geq1.
\end{equation*}
Thus
\begin{equation*}
  \frac{\pi (a^2+b^2) t_n^2 \log n}{4}\geq \frac{C_\mu(\theta+\mu-2)(\beta_0-\varepsilon)^2}{\theta} t_n^{2(\mu-2)},
\end{equation*}
letting $n\to +\infty$,
we get a contradiction. This ends the proof.
\end{proof}

\section{{\bfseries The monotonicity of $a\mapsto m(a,b)$ and $b\mapsto m(a,b)$}}\label{mono}
To guarantee the weak limit of a $(PS)_{m(a,b)}$ sequence is a ground state solution of problem \eqref{problem}-\eqref{problem'}, in this section, we investigate the monotonicity of the functions $a\mapsto m(a,b)$ and $b\mapsto m(a,b)$.
\begin{lemma}\label{6.1}
Assume that $(F_1)$-$(F_3)$ and $(F_6)$ hold, then the functions $a\mapsto m(a,b)$ and $b\mapsto m(a,b)$ are non-increasing on $(0,+\infty)$.
\end{lemma}

\begin{proof}
For any given $a,b>0$, if $\hat{a}>a$ and $\hat{b}>b$, we prove that $m(\hat{a},b)\leq m(a,b)$ and $m(a,\hat{b})\leq m(a,b)$.
By the definition of $m(a,b)$, for any $\delta>0$, there exists $(u,v)\in\mathcal{P}(a,b)$ such that
\begin{align}\label{6.1}
\mathcal{J}(u,v)\leq m(a,b)+\frac{\delta}{3}.
\end{align}
Consider a cut-off function $\varrho\in C_0^{\infty}(\mathbb{R},[0,1])$ such that $\varrho(x)=1$ if $|x|\leq 1$ and $\varrho(x)=0$ if $|x|\geq 2$. For any $\varepsilon>0$ small, define
\begin{align*}
u_{\varepsilon}(x):=\varrho(\varepsilon x)u(x)\in H^{1/2}(\mathbb{R})\backslash\{0\},
\end{align*}
then $(u_\varepsilon,v)\to (u,v)$ in $\mathcal{X}$ as $\varepsilon\to0^+$. From Lemmas $\ref{m}$ and $\ref{pa}$, we have $\beta_{(u_\varepsilon,v)}\to \beta_{(u,v)}=0$ in $\mathbb{R}$ and $\mathcal{H}((u_\varepsilon,v),\beta_{(u_\varepsilon,v)})\to \mathcal{H}((u,v),\beta_{(u,v)})=(u,v)$ in $\mathcal{X}$ as $\varepsilon\to0^+$.
Fix $\varepsilon_0>0$ small enough such that
\begin{align}\label{6.2}
\mathcal{ J}(\mathcal{H}((u_{\varepsilon_0},v), \beta_{(u_{\varepsilon_0},v)}))\leq \mathcal{J}(u,v)+\frac{\delta}{3}.
\end{align}
Let $v\in C_0^{\infty}(\mathbb{R})$ satisfy $supp (v)\subset B_{1+\frac{4}{\varepsilon_0}}(0)\backslash B_{\frac{4}{\varepsilon_0}}(0)$, and set
\begin{align*}
v_{\varepsilon_0}=\frac{\hat{a}^2-\|u_{\varepsilon_0}\|_2^2}{\|v\|_2^2}v.
\end{align*}
Define $s_h:=u_{\varepsilon_0}+\mathcal{H}(v_{\varepsilon_0},h)$ for any $h<0$. Since $dist(u_{\varepsilon_0},\mathcal{H}(v_{\varepsilon_0},h))\geq \frac{2}{\varepsilon_0}>0$, we obtain $\|s_h\|_2^2=\hat{a}^2$, i.e., $(s_h,v)\in S(\hat{a})\times S(b)$.

We claim that  $\beta_{(s_h,v)}$ is bounded from above as $h\to-\infty$. Otherwise, by $(F_3)$, $(\ref{fff})$ and $(s_h,v)\to (u_{\varepsilon_0},v)\neq (0,0)$ a.e. in $\mathbb{R}$ as $h\to-\infty$, one has
\begin{align*}
0&\leq \lim_{n\to\infty}e^{-\beta_{(s_h,v)}}\mathcal{J}(\mathcal{H}((s_h,v), \beta_{(s_h,v)}))\\
&\leq \lim_{n\to\infty} \frac{1}{2}\Big[\|(-\Delta)^{1/4}s_h\|_2^2+\|(-\Delta)^{1/4}v\|_2^2-e^{(\theta+\mu-3)\beta_{(s_h,v)}}\int_{\mathbb{R}}(I_\mu*F(s_h,v))F(s_h,v)\mathrm{d}x
\Big]=-\infty,
\end{align*}
which leads to a contradiction.
Thus $\beta_{(s_h,v)}+h\to-\infty$ as $h\to-\infty$, by $(F_3)$, we get
\begin{align}\label{6.3}
\mathcal{J}(\mathcal{H}((v_{\varepsilon_0},0), \beta_{(s_h,v)}+h))\leq \frac{e^{\beta_{(s_h,v)}+h}}{2}\|(-\Delta)^{1/4}v_{\varepsilon_0}\|_2^2\to0,\quad \mbox{as}\ h\to-\infty.
\end{align}
We deduce from Lemma $\ref{pa}$ and $(\ref{6.1})$-$(\ref{6.3})$ that
\begin{align*}
m(\hat{a},b)\leq \mathcal{J}(\mathcal{H}((s_h,v), \beta_{(s_h,v)}))
=&\mathcal{J}(\mathcal{H}((u_{\varepsilon_0},v), \beta_{(s_h,v)}))+\mathcal{J}(\mathcal{H}(\mathcal{H}((v_{\varepsilon_0},0),h), \beta_{(s_h,v)}))\\
=& \mathcal{J}(\mathcal{H}((u_{\varepsilon_0},v), \beta_{(s_h,v)}))+\mathcal{J}(\mathcal{H}((v_{\varepsilon_0},0), \beta_{(s_h,v)}+h))\\
\leq& \mathcal{J}(\mathcal{H}((u_{\varepsilon_0},v), \beta_{(u_{\varepsilon_0},v)}))+\mathcal{J}(\mathcal{H}((v_{\varepsilon_0},0), \beta_{(s_h,v)}+h))\\
\leq& m(a,b)+\delta.
\end{align*}
By the arbitrariness of $\delta>0$, we deduce that $m(\hat{a},b)\leq m(a,b)$ for any $\hat{a}>a$. Similarly, we can prove $m(a,\hat{b})\leq m(a,b)$.
\end{proof}

\begin{lemma}\label{6.2}
Assume that $(F_1)$-$(F_3)$ and $(F_6)$ hold. Suppose that $(\ref{problem})$ possesses a ground state solution with $\lambda_1,\lambda_2<0$, then $m(a^*,b)< m(a,b)$ for any $a^*>a$ close to $a$, and $m(a,b^*)< m(a,b)$ for any $b^*>b$ close to $b$.
\end{lemma}

\begin{proof}
For any $t>0$ and $\beta\in\mathbb{R}$, one has $\mathcal{H}((tu,v),\beta)\in S(ta)\times S(b)$ and
\begin{align*}
\mathcal{J}(\mathcal{H}((tu,v),\beta))=&\frac{e^\beta}{2}\big(t^2\|(-\Delta)^{1/4}u\|_2^2+\|(-\Delta)^{1/4}v\|_2^2\big)\\&-\frac{e^{(\mu-2)\beta}}{2}
\int_{\mathbb{R}}(I_\mu*F(te^\frac{\beta}{2}u,e^\frac{\beta}{2}v))F(te^\frac{\beta}{2}u,e^\frac{\beta}{2}v)\mathrm{d}x.
\end{align*}
Then
\begin{align*}
\frac{\partial{\mathcal{J}(\mathcal{H}((tu,v),\beta))}}{\partial t}&=te^\beta \|(-\Delta)^{1/4}u\|_2^2-e^{(\mu-2)\beta}
\int_{\mathbb{R}}(I_\mu*F(te^\frac{\beta}{2}u,e^\frac{\beta}{2}v))\frac{\partial F(te^\frac{\beta}{2}u,e^\frac{\beta}{2}v)}{\partial (te^\frac{\beta}{2}u)}e^{\frac{\beta}{2}}u\mathrm{d}x=:\frac{M}{t},
\end{align*}
where
\begin{align*}
  M=&\langle \mathcal{J}'(\mathcal{H}((tu,v),\beta)),\mathcal{H}((tu,v),\beta) \rangle-e^\beta \|(-\Delta)^{1/4}v\|_2^2\\&+e^{(\mu-2)\beta}
\int_{\mathbb{R}}(I_\mu*F(te^\frac{\beta}{2}u,e^\frac{\beta}{2}v))\frac{\partial F(te^\frac{\beta}{2}u,e^\frac{\beta}{2}v)}{\partial (e^\frac{\beta}{2}v)}e^{\frac{\beta}{2}}v\mathrm{d}x
\end{align*}
For convenience, we denote $\tau(t,\beta):=\mathcal{J}(\mathcal{H}((tu,v),\beta))$.
By Lemma $\ref{m}$, $\mathcal{H}((tu,v),\beta)\to (u,v)$ in $\mathcal{X}$ as $(t,\beta)\to(1,0)$.
Since $\lambda_1<0$, we have
\begin{equation*}
\langle \mathcal{J}'(u,v),(u,v) \rangle- \|(-\Delta)^{1/4}v\|_2^2+\int_{\mathbb{R}}(I_\mu*F(u,v))F_v(u,v)v \mathrm{d}x=\lambda_1\|u\|_2^2=\lambda_1 a^2<0.
\end{equation*}
Hence, for $\delta>0$ small enough, one has
\begin{align*}
\frac{\partial{\tau(t,\beta)}}{\partial t}<0, \quad\mbox{for any}\ (t,\beta)\in (1,1+\delta]\times[-\delta,\delta].
\end{align*}
For any $t\in(1,1+\delta]$ and $\beta\in [-\delta,\delta]$, using the mean value theorem, we obtain
\begin{align*}
\tau(t,\beta)=\tau(1,\beta)+(t-1)\cdot\frac{\partial{\tau(t,\beta)}}{\partial t}\Big|_{t=\xi}<\tau(1,\beta).
\end{align*}
for some $\xi\in(1,t)$.
By Lemma $\ref{pa}$, $\beta_{(tu,v)}\to\beta_{(u,v)}=0$ in $\mathbb{R}$ as $t\to 1^+$. For any $a^*>a$ close to $a$, let $\hat t=\frac{a^*}{a}$, then
$\hat t\in (1,1+\delta]$ and $\beta_{(\hat t u,v)}\in[-\delta,\delta]$.
Applying Lemma $\ref{pa}$ again, we have
\begin{align*}
m(a^*,b)\leq \tau(\hat t,\beta_{(\hat t u,v)})<\tau(1,\beta_{(\hat tu,v)})=\mathcal{J}(\mathcal{H}((u,v), \beta_{(\hat t u,v})))\leq \mathcal{J}(u,v)=m(a,b).
\end{align*}
Analogously, we can prove that $m(a,b^*)< m(a,b)$ for any $b^*>b$ close to $b$.
\end{proof}
From Lemmas \ref{6.1} and \ref{6.2}, we immediately have the following result.
\begin{lemma}\label{6.3}
Assume that $(F_1)$-$(F_3)$ and $(F_6)$ hold. Suppose that $(\ref{problem})$ possesses a ground state solution with $\lambda_1,\lambda_2<0$, then $a\mapsto m(a,b)$ and $b\mapsto m(a,b)$ are decreasing on $(0,+\infty)$. 
\end{lemma}

\section{{\bfseries Palais-Smale sequence}}\label{ps}
In this section, using the minimax principle based on the homotopy stable family of compact subsets of $\mathcal{S}$ (see \cite{ghou} for more details), we construct a $(PS)_{m(a,b)}$ sequence on $\mathcal{P}(a,b)$ for $\mathcal{J}|_{\mathcal{S}}$.

\begin{proposition}\label{pro}
Assume that $(F_1)$-$(F_3)$ and $(F_6)$ hold, then there exists a $(PS)_{m(a,b)}$ sequence $\{(u_n,v_n)\}\subset\mathcal{P}(a,b)$ for $\mathcal{J}|_{\mathcal{S}}$.
\end{proposition}

Following by \cite{willem}, we recall that the tangent space of $\mathcal{S}$ at $(u,v)$ is defined by
\begin{align*}
T_{(u,v)}:=\Big\{(\varphi,\psi)\in \mathcal{X} : \int_{\mathbb{R}}(u\varphi+v \psi)\mathrm{d}x=0\Big\}.
\end{align*}
To prove Proposition $\ref{pro}$, we borrow some arguments from \cite{BS2,soave1} and consider the functional $\mathcal{I}: \mathcal{S}\to\mathbb{R}$ defined by
\begin{align*}
\mathcal{I}(u,v)=\mathcal{J}(\mathcal{H}((u,v),\beta_{(u,v)})),
\end{align*}
where $\beta_{(u,v)}\in\mathbb{R}$ is the unique number obtained in Lemma \ref{pa} for any $(u,v)\in \mathcal{S}$. By Lemma \ref{pa}, we know that $\beta_{(u,v)}$ is continuous as a mapping for any $(u,v)\in \mathcal{S}$. However, it remains unknown that whether $\beta_{(u,v)}$ is of class $C^1$. Inspired by \cite[Proposition 2.9]{sz}, we have
\begin{lemma}\label{c1}
Assume that $(F_1)-(F_3)$ and $(F_6)$ hold, then the functional $\mathcal{I}: \mathcal{S}\to \mathbb{R}$ is of class $C^1$ and
\begin{align*}
&\langle \mathcal{I}^{\prime}(u,v), (\varphi,\psi)\rangle\\=&\frac{e^{\beta_{(u,v)}}}{2\pi}\int_{\mathbb{R}}\int_{\mathbb{R}}\frac{|u(x)-u(y)||\varphi(x)-\varphi(y)|}{|x-y|^2}\mathrm{d}x\mathrm{d}y+\frac{e^{\beta_{(u,v)}}}{2\pi}\int_{\mathbb{R}}\int_{\mathbb{R}}\frac{|v(x)-v(y)||\psi(x)-\psi(y)|}{|x-y|^2}\mathrm{d}x\mathrm{d}y
\\&-e^{(\mu-2)\beta_{(u,v)}}\int_{\mathbb{R}}(I_\mu*F(e^{\frac{\beta_{(u,v)}}{2}}u,e^{\frac{\beta_{(u,v)}}{2}}v))
\\ &\times \bigg[\big(e^{\frac{\beta_{(u,v)}}{2}}\varphi,e^{\frac{\beta_{(u,v)}}{2}}\psi\big)\Big(\frac{\partial F(e^{\frac{\beta_{(u,v)}}{2}}u,e^{\frac{\beta_{(u,v)}}{2}}v)}{\partial (e^{\frac{\beta_{(u,v)}}{2}}u)},\frac{\partial F(e^{\frac{\beta_{(u,v)}}{2}}u,e^{\frac{\beta_{(u,v)}}{2}}v)}{\partial (e^{\frac{\beta_{(u,v)}}{2}}v)}\Big)\bigg]\\
=&\langle \mathcal{J}^{\prime}(\mathcal{H}((u,v),\beta_{(u,v)})) ,\mathcal{H}((\varphi,\psi),\beta_{(\varphi,\psi)})\rangle
\end{align*}
for any $(u,v)\in \mathcal{S}$ and $(\varphi,\psi)\in T_{(u,v)}$.
\end{lemma}
\begin{proof}
Let $(u,v)\in \mathcal{S}$ and $(\varphi,\psi)\in T_{(u,v)}$, for any $|t|$ small enough, by Lemma $\ref{pa}$,
\begin{align*}
&\mathcal{I}(u+t\varphi,v+t\psi)-\mathcal{I}(u,v)\\=&\mathcal J(\mathcal{H}((u+t\varphi,v+t\psi),s_{(u+t\varphi,v+t\psi)}))-\mathcal J(\mathcal{H}((u,v),s_{(u,v)}))\\
\leq& \mathcal J(\mathcal{H}((u+t\varphi,v+t\psi),s_{(u+t\varphi,v+t\psi)}))-\mathcal J(\mathcal{H}((u,v),s_{(u+t\varphi,v+t\psi)}))\\
=&\frac{1}{2}e^{\beta_{(u+t\varphi,v+t\psi)}}\Big[\|(-\Delta)^{1/4}(u+t\varphi)\|_2^2-\|(-\Delta)^{1/4}u\|_2^2+\|(-\Delta)^{1/4}(v+t\psi)\|_2^2-\|(-\Delta)^{1/4}v\|_2^2
\Big]\\
&-\frac{1}{2}e^{(\mu-2)\beta_{(u+t\varphi,v+t\psi)}}\int_{\mathbb{R}}\Big[(I_\mu*F(e^{\frac{\beta_{(u+t\varphi,v+t\psi)}}{2}}(u+t\varphi),e^{\frac{\beta_{(u+t\varphi,v+t\psi)}}{2}}(v+t\psi)))\\& \times F(e^{\frac{\beta_{(u+t\varphi,v+t\psi)}}{2}}(u+t\varphi),e^{\frac{\beta_{(u+t\varphi,v+t\psi)}}{2}}(v+t\psi))\\
&-(I_\mu*F(e^{\frac{\beta_{(u+t\varphi,v+t\psi)}}{2}}u,e^{\frac{\beta_{(u+t\varphi,v+t\psi)}}{2}}v))F(e^{\frac{\beta_{(u+t\varphi,v+t\psi)}}{2}}u,e^{\frac{\beta_{(u+t\varphi,v+t\psi)}}{2}}v)
\Big]\mathrm{d}x\\
=&\frac{ e^{\beta_{(u+t\varphi,v+t\psi)}}}{2}\Big[t^2\|(-\Delta)^{1/4}\varphi\|_2^2+t^2\|(-\Delta)^{1/4}\psi\|_2^2\\
&+\frac{2t}{2\pi}\int_{\mathbb{R}}\int_{\mathbb{R}}\frac{|u(x)-u(y)||\varphi(x)-\varphi(y)|}{|x-y|^2}\mathrm{d}x\mathrm{d}y+\frac{2t}{2\pi}\int_{\mathbb{R}}\int_{\mathbb{R}}\frac{|v(x)-v(y)||\psi(x)-\psi(y)|}{|x-y|^2}\mathrm{d}x\mathrm{d}y\Big]\\
  & -\frac{e^{(\mu-2)\beta_{(u+t\varphi,v+t\psi)}}}{2}\int_{\mathbb{R}}(I_\mu*F(e^{\frac{\beta_{(u+t\varphi,v+t\psi)}}{2}}(u+t\varphi),e^{\frac{\beta_{(u+t\varphi,v+t\psi)}}{2}}(v+t\psi)))\\&
  \times \Big[\big(e^{\frac{\beta_{(u+t\varphi,v+t\psi)}}{2}}t\varphi,e^{\frac{\beta_{(u+t\varphi,v+t\psi)}}{2}}t\psi\big)\cdot \big(F_{z_1}|_{z_1=e^{\frac{\beta_{(u+t\varphi,v+t\psi)}}{2}}(u+\xi_tt\varphi)},F_{z_2}|_{z_2=e^{\frac{\beta_{(u+t\varphi,v+t\psi)}}{2}}(v+\xi_tt\psi)} \big) \Big]\mathrm{d}x\\&  -\frac{e^{(\mu-2)s_{(u+t\varphi,v+t\psi)}}}{2}\int_{\mathbb{R}}(I_\mu*F(e^{\frac{\beta_{(u+t\varphi,v+t\psi)}}{2}}u,e^{\frac{\beta_{(u+t\varphi,v+t\psi)}}{2}}v))\\&
  \times \Big[\big(e^{\frac{\beta_{(u+t\varphi,v+t\psi)}}{2}}t\varphi,e^{\frac{\beta_{(u+t\varphi,v+t\psi)}}{2}}t\psi\big)\cdot \big(F_{z_1}|_{z_1=e^{\frac{\beta_{(u+t\varphi,v+t\psi)}}{2}}(u+\xi_tt\varphi)},F_{z_2}|_{z_2=e^{\frac{\beta_{(u+t\varphi,v+t\psi)}}{2}}(v+\xi_tt\psi)} \big) \Big] \mathrm{d}x,
\end{align*}
where $\xi_t\in(0,1)$. On the other hand,
\begin{align*}
&\mathcal{I}(u+t\varphi,v+t\psi)-\mathcal{I}(u,v)\\ \geq &\mathcal J(\mathcal{H}((u+t\varphi,v+t\psi),s_{(u,v)}))-\mathcal J(\mathcal{H}((u,v),s_{(u,v)}))\\ \geq &\frac{ e^{\beta_{(u,v)}}}{2}\Big[t^2\|(-\Delta)^{1/4}\varphi\|_2^2+t^2\|(-\Delta)^{1/4}\psi\|_2^2\\
&+\frac{2t}{2\pi}\int_{\mathbb{R}}\int_{\mathbb{R}}\frac{|u(x)-u(y)||\varphi(x)-\varphi(y)|}{|x-y|^2}\mathrm{d}x\mathrm{d}y+\frac{2t}{2\pi}\int_{\mathbb{R}}\int_{\mathbb{R}}\frac{|v(x)-v(y)||\psi(x)-\psi(y)|}{|x-y|^2}\mathrm{d}x\mathrm{d}y\Big]\\
  & -\frac{e^{(\mu-2)\beta_{(u,v)}}}{2}\int_{\mathbb{R}}(I_\mu*F(e^{\frac{\beta_{(u,v)}}{2}}(u+t\varphi),e^{\frac{\beta_{(u,v)}}{2}}(v+t\psi)))\\&
  \times \Big[\big(e^{\frac{\beta_{(u,v)}}{2}}t\varphi,e^{\frac{\beta_{(u,v)}}{2}}t\psi\big)\cdot \big(F_{z_1}|_{z_1=e^{\frac{\beta_{(u,v)}}{2}}(u+\xi_tt\varphi)},F_{z_2}|_{z_2=e^{\frac{\beta_{(u,v)}}{2}}(v+\xi_tt\psi)} \big) \Big]\mathrm{d}x\\&  -\frac{e^{(\mu-2)s_{(u,v)}}}{2}\int_{\mathbb{R}}(I_\mu*F(e^{\frac{\beta_{(u,v)}}{2}}u,e^{\frac{\beta_{(u,v)}}{2}}v))\\&
  \times \Big[\big(e^{\frac{\beta_{(u,v)}}{2}}t\varphi,e^{\frac{\beta_{(u,v)}}{2}}t\psi\big)\cdot \big(F_{z_1}|_{z_1=e^{\frac{\beta_{(u,v)}}{2}}(u+\xi_tt\varphi)},F_{z_2}|_{z_2=e^{\frac{\beta_{(u,v)}}{2}}(v+\xi_tt\psi)} \big) \Big] \mathrm{d}x,
\end{align*}
where $\zeta_t\in(0,1)$. By Lemma $\ref{pa}$, $\lim\limits_{t\to0}\beta_{(u+t\varphi,v+t\psi)}=\beta_{(u,v)}$, from the above inequalities, we conclude
\begin{align*}
&\lim_{t\to0}\frac{\mathcal{I}(u+t\varphi,v+t\psi)-\mathcal{I}(u,v)}{t}\\=&\frac{e^{\beta_{(u,v)}}}{2\pi}\int_{\mathbb{R}}\int_{\mathbb{R}}\frac{|u(x)-u(y)||\varphi(x)-\varphi(y)|}{|x-y|^2}\mathrm{d}x\mathrm{d}y+\frac{e^{\beta_{(u,v)}}}{2\pi}\int_{\mathbb{R}}\int_{\mathbb{R}}\frac{|v(x)-v(y)||\psi(x)-\psi(y)|}{|x-y|^2}\mathrm{d}x\mathrm{d}y
\\&-e^{(\mu-2)\beta_{(u,v)}}\int_{\mathbb{R}}(I_\mu*F(e^{\frac{\beta_{(u,v)}}{2}}u,e^{\frac{\beta_{(u,v)}}{2}}v))
\\ &\times \bigg[\big(e^{\frac{\beta_{(u,v)}}{2}}\varphi,e^{\frac{\beta_{(u,v)}}{2}}\psi\big)\Big(\frac{\partial F(e^{\frac{\beta_{(u,v)}}{2}}u,e^{\frac{\beta_{(u,v)}}{2}}v)}{\partial (e^{\frac{\beta_{(u,v)}}{2}}u)},\frac{\partial F(e^{\frac{\beta_{(u,v)}}{2}}u,e^{\frac{\beta_{(u,v)}}{2}}v)}{\partial (e^{\frac{\beta_{(u,v)}}{2}}v)}\Big)\bigg].
\end{align*}
Using Lemma $\ref{pa}$ again, we find that the G\^ateaux derivative of $\mathcal{I}$ is continuous linear in $(\varphi,\psi)$ and continuous in $(u,v)$.
Therefore, by \cite[Proposition 1.3]{willem}, we obtain $\mathcal{I}$ is of class $C^1$. Changing variables in the integrals, we prove the rest.
\end{proof}

\begin{lemma}\label{Minimax}
Assume that $(F_1)$-$(F_3)$ and $(F_6)$ hold. Let $\mathcal{F}$ be a homotopy stable family of compact subsets of $\mathcal{S}$ without boundary and set
\begin{align*}
m_{\mathcal{F}}:=\inf_{A\in\mathcal{F}}\max_{(u,v)\in A}\mathcal{I}(u,v).
\end{align*}
If $m_\mathcal{F}>0$, then there exists a $(PS)_{m_{\mathcal{F}}}$ sequence $\{(u_n,v_n)\}\subset \mathcal{P}(a,b)$ for $\mathcal{J}|_{\mathcal{S}}$.
\end{lemma}
\begin{proof}
Let $\{A_n\}\subset \mathcal{F}$ be a minimizing sequence of $m_{\mathcal{F}}$. We define the mapping $\eta : [0,1]\times \mathcal{S}\to \mathcal{S}$, that is $\eta(t,(u,v))=\mathcal{H}((u,v), t\beta_{(u,v)})$. By Lemmas \ref{m} and $\ref{pa}$, $\eta(t, (u,v))$ is continuous in $[0,1]\times \mathcal{S}$ and satisfies $\eta(t,(u,v))=(u,v)$ for all $(t,(u,v))\in \{0\}\times \mathcal{S}$. Thus by the definition of $\mathcal{F}$ (see \cite[Definition 3.1]{ghou}), one has
\begin{align*}
Q_n:=\eta(1, A_n)=\{\mathcal{H}((u,v),\beta_{(u,v)}) : (u,v)\in A_n\}\subset \mathcal{F}.
\end{align*}
Obviously, $Q_n\subset \mathcal{P}(a,b)$ for any $n\in\mathbb{N}^+$. Since $\mathcal{I}(\mathcal{H}((u,v),\beta))=\mathcal{I}(u,v)$ for any $(u,v)\in \mathcal{S}$ and $\beta\in\mathbb{R}$, then
\begin{align*}
\max_{(u,v)\in Q_n}\mathcal{I}(u,v)=\max_{(u,v)\in A_n}\mathcal{I}(u,v)\to m_{\mathcal{F}},\quad \mbox{as $n\rightarrow\infty$},
\end{align*}
which implies that $\{Q_n\}\subset \mathcal{F}$ is another minimizing sequence of $m_{\mathcal{F}}$. Since $G_1(u):=\|u\|_2^2-a^2$, $G_2(v):=\|v\|_2^2-b^2$ are of class $C^1$, and for any $(u,v)\in \mathcal{S}$, we have $\langle G'_1(u),u\rangle=2a^2>0$, $\langle G'_2(v),v\rangle=2b^2>0$. Therefore, by the implicit function theorem, $\mathcal{S}$ is a $C^1$-Finsler manifold.  By \cite[Theorem 3.2]{ghou}, we obtain a $(PS)_{m_{\mathcal{F}}}$ sequence $\{(\hat{u}_n,\hat{v}_n)\}\subset \mathcal{S}$ for $\mathcal{I}$ such that $\lim\limits_{n\to+\infty}dist((\hat{u}_n,\hat{v}_n), Q_n)=0$.
Let
\begin{align*}
(u_n,v_n):=\mathcal{H}((\hat{u}_n,\hat{v}_n),\beta_{(\hat{u}_n,\hat{v}_n)}),
\end{align*}
we prove that $\{(u_n,v_n)\}\subset \mathcal{P}(a,b)$ is the desired sequence.
We claim that there exists $C>0$ such that $e^{-\beta_{(u_n,v_n)}}\leq C$ for any $n\in\mathbb{N}^+$. Indeed, we have
\begin{align*}
e^{-\beta_{(u_n,v_n)}}=\frac{\|(-\Delta)^{1/4}\hat{u}_n\|_2^2+\|(-\Delta)^{1/4}\hat{v}_n\|_2^2}{\|(-\Delta)^{1/4}u_n\|_2^2+\|(-\Delta)^{1/4}v_n\|_2^2}.
\end{align*}
Since $\{(u_n,v_n)\}\subset \mathcal{P}(a)$, by Lemma $\ref{ous1}$, we know that there exists a constant $C>0$ such that $\|(-\Delta)^{1/4}u_n\|_2^2+\|(-\Delta)^{1/4}v_n\|_2^2\geq C$ for any $n\in\mathbb{N}^+$. Since $Q_n\subset\mathcal{P}(a,b)$ for any $n\in\mathbb{N}^+$ and for any $(u,v)\in\mathcal{P}(a,b)$, one has $\mathcal{J}(u,v)=\mathcal{I}(u,v)$, then
\begin{align*}
\max_{(u,v)\in Q_n}\mathcal{J}(u,v)=\max_{(u,v)\in Q_n}\mathcal{I}(u,v)\to m_{\mathcal{F}}, \quad \mbox{as} \ n\to+\infty.
\end{align*}
This fact together with $Q_n\subset \mathcal{P}(a,b)$ and $(F_3)$ yields that $\{Q_n\}$ is uniformly bounded in $\mathcal{X}$,
thus from $\lim\limits_{n\to \infty}dist((\hat{u}_n,\hat{v}_n),Q_n)=0$, we obtain $\sup\limits_{n\geq 1}\|(\hat{u}_n,\hat{v}_n)\|^2<+\infty$. This prove the claim.

Since $\{(u_n,v_n)\}\subset \mathcal{P}(a,b)$, one has $\mathcal{J}(u_n,v_n)=\mathcal{I}(u_n,v_n)=\mathcal{I}(\hat{u}_n,\hat{v}_n)\to m_{\mathcal{F}}$ as $n\to\infty$. For any $(\varphi,\psi)\in T_{(u_n,v_n)}$, we have
\begin{align*}
  &\int_{\mathbb{R}}\big(\hat{u}_n e^{-\frac{\beta_{(\hat{u}_n,\hat{v}_n)}}{2}}\varphi(e^{-\beta_{(\hat{u}_n,\hat{v}_n)}}x) +\hat{v}_n e^{-\frac{\beta_{(\hat{u}_n,\hat{v}_n)}}{2}}\psi(e^{-\beta_{(\hat{u}_n,\hat{v}_n)}}x)\big) dx\\
=&\int_{\mathbb{R}}\big(\hat{u}_n(e^{\beta_{(\hat{u}_n,\hat{v}_n)}}y) e^{\frac{\beta_{(\hat{u}_n,\hat{v}_n)}}{2}}\varphi(y) +\hat{v}_n (e^{\beta_{(\hat{u}_n,\hat{v}_n)}}y) e^{\frac{\beta_{(\hat{u}_n,\hat{v}_n)}}{2}}\psi(y)\big)dy=\int_{\mathbb{R}}(u_n\varphi+v_n\psi)dx=0,
\end{align*}
which implies that $\mathcal{H}((\varphi,\psi),-\beta_{(\hat{u}_n,\hat{v}_n)})\in T_{(u_n,v_n)}\mathcal{S}$. Also,
\begin{align*}
 &\big\|(e^{-\frac{\beta_{(\hat{u}_n,\hat{v}_n)}}{2}}\varphi(e^{-\beta_{(\hat{u}_n,\hat{v}_n)}}x) ,e^{-\frac{\beta_{(\hat{u}_n,\hat{v}_n)}}{2}}\psi(e^{-\beta_{(\hat{u}_n,\hat{v}_n)}}x) )\big\|^2\\
=&e^{-\beta_{(\hat{u}_n,\hat{v}_n)}}\big(\|(-\Delta)^{1/4}\varphi\|_2^2+\|(-\Delta)^{1/4}\psi\|_2^2\big)+\|\varphi\|_2^2+\|\psi\|_2^2\\
\leq &C\big(\|(-\Delta)^{1/4}\varphi\|_2^2+\|(-\Delta)^{1/4}\psi\|_2^2\big)+\|\varphi\|_2^2+\|\psi\|_2^2\\
\leq &\max\{1,C\}\|\varphi,\psi\|^2.
\end{align*}
By Lemma $\ref{c1}$, for any $(\varphi,\psi)\in T_{(u_n,v_n)}$, we deduce that
\begin{align*}
\big|\langle \mathcal{J}^{\prime}(u_n,v_n),(\varphi,\psi)\rangle\big|=&\Big|\big\langle \mathcal{J}^{\prime}\big(\mathcal{H}((\hat{u}_n,\hat{v}_n),\beta_{(\hat{u}_n,\hat{v}_n)})\big), \mathcal{H}\big(\mathcal{H}((\varphi,\psi),-\beta_{(\hat{u}_n,\hat{v}_n)}),\beta_{(\hat{u}_n,\hat{v}_n)}\big)\big\rangle\Big|\\
=&\Big|\big\langle \mathcal{I}^{\prime}(\hat{u}_n,\hat{v}_n),\mathcal{H}((\varphi,\psi),-\beta_{(\hat{u}_n,\hat{v}_n)})\Big\rangle\Big|\\
\leq&\|\mathcal{I}^{\prime}(\hat{u}_n,\hat{v}_n)\|_*\cdot\|\mathcal{H}((\varphi,\psi),-\beta_{(\hat{u}_n,\hat{v}_n)})\|\\
\leq &\max\big\{1, \sqrt{C}\big\}\|\mathcal{I}^{\prime}(\hat{u}_n,\hat{v}_n)\|_*\cdot\|\varphi,\psi\|,
\end{align*}
where $(\mathcal{X}^*,\|\cdot\|_*)$ is the dual space of $(\mathcal{X},\|\cdot\|)$. Hence we can deduce that
\begin{align*}
\|\mathcal{J}^{\prime}(u_n,v_n)\|_* \leq\max\big\{1, \sqrt{C}\big\}\|\mathcal{I}^{\prime}(\hat{u}_n,\hat{v}_n)\|_*\to0, \quad \mbox{as}\ n\to\infty,
\end{align*}
 which implies that  $\{(u_n,v_n)\}$ is a $(PS)_{m_{\mathcal{F}}}$ sequence for $\mathcal{J}|_{\mathcal{S}}$. This ends the proof.
\end{proof}

\noindent{\bfseries Proof of Proposition \ref{pro}.}
Note that the class $\mathcal{F}$ of all singletons included in $\mathcal{S}$ is a homotopy stable family of compact subsets of $\mathcal{S}$ without boundary.
By Lemma $\ref{Minimax}$, we know that if $m_{\mathcal{F}}>0$, then there exists a $(PS)_{m_{\mathcal{F}}}$ sequence $\{(u_n,v_n)\}\subset \mathcal{P}(a,b)$ for $\mathcal{J}|_{\mathcal{S}}$.
By Lemma \ref{ous1}, we know $m(a,b)>0$, so if we can prove that $m_{\mathcal{F}}=m(a,b)$, then we complete the proof.

In fact,
by the definition of $\mathcal{F}$, we have
\begin{align*}
m_{\mathcal{F}}=\inf_{A\in\mathcal{F}}\max_{(u,v)\in A}\mathcal{I}(u,v)=\inf_{(u,v)\in \mathcal{S}}\mathcal{I}(u,v)=\inf_{(u,v)\in \mathcal{S}}\mathcal{I}(\mathcal{H}((u,v),\beta_{(u,v)}))=\inf_{(u,v)\in \mathcal{S}}\mathcal{J}(\mathcal{H}((u,v),\beta_{(u,v)})).
\end{align*}
For any $(u,v)\in \mathcal{S}$, it follows from $\mathcal{H}((u,v),\beta_{(u,v)})\in \mathcal{P}(a,b)$ that $\mathcal{J}(\mathcal{H}((u,v),\beta_{(u,v)}))\geq m(a,b)$, so $m_{\mathcal{F}}\geq m(a,b)$.
On the other hand, for any $(u,v)\in \mathcal{P}(a,b)$, by Lemma $\ref{pa}$, we deduce that $\beta_{(u,v)}=0$ and $\mathcal{J}(u,v)=\mathcal{J}(\mathcal{H}((u,v),0))\geq \inf\limits_{(u,v)\in \mathcal{S}}\mathcal{J}(\mathcal{H}((u,v),\beta_{(u,v)}))$, which implies that $m(a,b)\geq m_{\mathcal{F}}$.
  \qed

For the sequence $\{(u_n,v_n)\}$ obtained in Proposition $\ref{pro}$, by $(F_3)$, we know that $\{(u_n,v_n)\}$ is bounded in $\mathcal{X}$. Up to a subsequence, we assume that $(u_n,v_n)\rightharpoonup (u,v)$ in $\mathcal{X}$. Furthermore, by $\mathcal{J}\big|^{\prime}_{\mathcal{S}}(u_n,v_n)\to0$ as $n\to+\infty$ and the Lagrange multiplier rule, there exist two sequences $\{\lambda_{1,n}\},\{\lambda_{2,n}\}\subset \mathbb{R}$ such that
\begin{align}\label{key}
&\frac{1}{2 \pi}\int_{\mathbb{R}}\int_{\mathbb{R}}\frac{[u_n(x)-u_n(y)][\varphi(x)-\varphi(y)]}{|x-y|^{2}}\mathrm{d}x\mathrm{d}y+\frac{1}{2 \pi}\int_{\mathbb{R}}\int_{\mathbb{R}}\frac{[v_n(x)-v_n(y)][\psi(x)-\psi(y)]}{|x-y|^{2}}\mathrm{d}x\mathrm{d}y \nonumber\\
&-\int_{\mathbb{R}}(I_\mu*F(u_n,v_n))F_{u_n}(u_n,v_n)\varphi \mathrm{d}x-\int_{\mathbb{R}}(I_\mu*F(u_n,v_n))F_{v_n}(u_n,v_n)\phi \mathrm{d}x \nonumber\\
=& \int_{\mathbb{R}}(\lambda_{1,n}u_n\varphi+\lambda_{2,n}v_n\psi)dx+o_n(1)\|(\varphi,\psi)\|
\end{align}
for any $(\varphi,\psi)\in \mathcal{X}$.

\begin{lemma}\label{bdd}
Assume that $(F_1)$-$(F_4)$ and $(F_6)$ hold, then $\{\lambda_{1,n}\}$ and $\{\lambda_{2,n}\}$ are bounded in $\mathbb{R}$.
\end{lemma}
\begin{proof}
Using $(u_n, 0)$ and $(0, v_n)$ as test functions in \eqref{key}, we have
\begin{equation}\label{bdd1}
  \lambda_{1,n}a^2=\|(-\Delta)^{1/4}u_n\|_2^2-\int_{\mathbb{R}}(I_\mu*F(u_n,v_n)) F_{u_n}(u_n,v_n)u_n dx+o_n(1)
\end{equation}
and
\begin{equation}\label{bdd2}
  \lambda_{2,n}b^2=\|(-\Delta)^{1/4}v_n\|_2^2-\int_{\mathbb{R}}(I_\mu*F(u_n,v_n)) F_{v_n}(u_n,v_n)v_n dx+o_n(1).
\end{equation}
By $(F_3)$-$(F_4)$, $P(u_n,v_n)=0$,
and the boundedness of $\{(u_n,v_n)\}$, we get $\{\lambda_{1,n}\}$ and $\{\lambda_{2,n}\}$ are bounded in $\mathbb{R}$. Up to a subsequence, we assume that $\lambda_{1,n}\rightarrow\lambda_1$ and $\lambda_{2,n}\rightarrow\lambda_2$ in $\mathbb{R}$ as $n\rightarrow\infty$.
\end{proof}

\begin{lemma}\label{non}
Assume that $(F_1)$-$(F_7)$ hold, then up to a subsequence and up to translations in $\mathbb{R}$, $u_a\neq0$ and $v_b\neq0$.
\end{lemma}

\begin{proof}
We claim that
\begin{align*}
\Lambda:=\limsup_{n\to+\infty}\Big(\sup_{y\in\mathbb{R}}\int_{B_r(y)}(|u_n|^2+|v_n|^2)\mathrm{d}x\Big)>0.
\end{align*}
Otherwise, we obtain $u_n,v_n\to0$ in $L^p(\mathbb{R})$ for any $p>2$ by the Lions' vanishing lemma \cite[Lemma 1.21]{willem}.
From $\mathcal{J}(u_n,v_n)=m(a,b)+o_n(1)$, $P(u_n,v_n)=0$ and $(F_3)$, we have
\begin{align*}
\mathcal{J}(u_n,v_n)-\frac{1}{2}P(u_n,v_n)\geq \frac{\theta+\mu-3}{2 \theta}\int_{\mathbb{R}}(I_\mu*F(u_n,v_n))\big[(u_n,v_n)\cdot(\nabla F(u_n,v_n))\big]\mathrm{d}x+o_n(1).
\end{align*}
By $\theta>3-\mu$, up to a subsequence, we get
\begin{align}\label{1n}
\int_{\mathbb{R}}(I_\mu*F(u_n,v_n))\big[(u_n,v_n)\cdot(\nabla F(u_n,v_n))\big]\mathrm{d}x\leq \frac{2\theta m(a,b)}{\theta+\mu-3}=:K_0.
\end{align}
From Lemma \ref{f}, we can see
\begin{align*}
\int_{\mathbb{R}}(I_\mu*F(u_n,v_n))F(u_n,v_n)\mathrm{d}x=o_n(1).
\end{align*}
Thus, by Lemma \ref{contr}, we have
\begin{align*}
\limsup_{n\to\infty}\big(\|(-\Delta)^{1/4}u_n\|_2^2+\|(-\Delta)^{1/4}v_n\|_2^2\big)\leq 2m(a,b)<\frac{2-\mu}{2}.
\end{align*}
Up to a subsequence, we assume that $\sup \limits_{n\in\mathbb{N}^+}\frac{2}{2-\mu}\big(\|(-\Delta)^{1/4}u_n\|_2^2+\|(-\Delta)^{1/4}v_n\|_2^2\big)<1$.
From $(\ref{tain1})$,  for $\nu'=\frac{\nu}{\nu-1}$, we have
\begin{align*}
\|F(u_n,v_n)\|_{\frac{2}{2-\mu}}^2\leq
C\Big(\|u_n\|_{\frac{2(\kappa+1)}{2-\mu}}^{\frac{2(\kappa+1)}{2-\mu}}+\|v_n\|_{\frac{2(\kappa+1)}{2-\mu}}^{\frac{2(\kappa+1)}{2-\mu}}\Big)^{2-\mu}+
C\Big(\|u_n\|_{\frac{2q\nu'}{2-\mu}}^{\frac{2q\nu'}{2-\mu}}+\|v_n\|_{\frac{2q\nu'}{2-\mu}}^{\frac{2q\nu'}{2-\mu}}\Big)^{\frac{2-\mu}{\nu'}}\to0,\  \mbox{as}\ n\to+\infty.
\end{align*}
By a similar argument as above, we infer that $ \|(u_n,v_n)\cdot \nabla F(u_n,v_n)\|_{\frac{2}{2-\mu}}\to0$ as $n\to\infty$.
Hence, we obtain
\begin{align*}
\int_{\mathbb{R}}(I_\mu*F(u_n,v_n))\big[(u_n,v_n)\cdot \nabla F(u_n,v_n)\big]\mathrm{d}x=o_n(1).
\end{align*}
Since $P(u_n,v_n)=0$, we have $\| (-\Delta)^{1/4}u_n\|_2^2+\| (-\Delta)^{1/4}v_n\|_2^2=o_n(1)$, then $m(a,b)=0$, which is a contradiction. According to $\Lambda>0$, there exists $\{y_n\}\subset\mathbb{R}$ such that $\int_{B_1(y_n)}(|u_n|^2+|y_n|^2)\mathrm{d}x>\frac{\Lambda}{2}$, i.e., $\int_{B_1(0)}(|u_n(x-y_n)|^2+|v_n(x-y_n)|^2)\mathrm{d}x>\frac{\Lambda}{2}$. Then up to a subsequence and up to translations in $\mathbb{R}$, $(u_n,v_n)\rightharpoonup (u_a,v_b)\neq(0,0)$ in $\mathcal{X}$.
By \eqref{key}, \eqref{1n}, Lemmas \ref{f} and \ref{bdd}, we can see that $(u_a,v_b)$ is a weak solution of \eqref{problem}. Assume that $u_a=0$, then by $(F_3)$ and $(F_5)$, we know $v_b=0$. Similarly, $v_b=0$ implies that $u_a=0$. This ends the proof.
\end{proof}

\begin{lemma}\label{ne}
Assume that $(F_1)$-$(F_7)$ hold, then $\lambda_1,\lambda_2<0$.
\end{lemma}

\begin{proof}
Combining \eqref{bdd1}, \eqref{bdd2} with  $P(u_n,v_n)=0$, we have
\begin{equation*}
  -\lambda_{1,n}a^2=\| (-\Delta)^{1/4}v_n\|_2^2+\int_{\mathbb{R}}(I_\mu*F(u_n,v_n))\big[(2-\mu)F(u_n,v_n)-F_{v_n}(u_n,v_n)v_n\big]\mathrm{d}x+o_n(1)
\end{equation*}
and
\begin{equation*}
  -\lambda_{2,n}b^2=\| (-\Delta)^{1/4}u_n\|_2^2+\int_{\mathbb{R}}(I_\mu*F(u_n,v_n))\big[(2-\mu)F(u_n,v_n)-F_{u_n}(u_n,v_n)u_n\big]\mathrm{d}x+o_n(1).
\end{equation*}
Thanks to $u_a\neq0$ and $v_b\neq0$, by using $(F_3)$-$(F_4)$ and Fatou lemma, we obtain $\liminf \limits_{n\rightarrow\infty}-\lambda_{1,n}>0$ and $\liminf\limits_{n\rightarrow\infty}-\lambda_{2,n}>0$, namely, $\limsup \limits_{n\rightarrow\infty}\lambda_{1,n}<0$ and $\limsup\limits_{n\rightarrow\infty}\lambda_{2,n}<0$. By Lemma \ref{bdd}, $\{\lambda_{1,n}\}$ and $\{\lambda_{2,n}\}$ are bounded in $\mathbb{R}$, up to a subsequence, we can assume that $\lambda_{1,n}\rightarrow\lambda_1<0$ and $\lambda_{2,n}\rightarrow\lambda_2<0$ in $\mathbb{R}$ as $n\rightarrow\infty$.
\end{proof}

\section{{\bfseries Proof of the result}}\label{proof}

\noindent{\bfseries Proof of Theorem \ref{thm1.1}.} Under the assumptions of Theorem \ref{thm1.1}, from \eqref{key}, \eqref{1n}, Lemmas \ref{f}, \ref{bdd}, \ref{ne},
we know $u$ is a weak solution of $(\ref{problem})$ with $\lambda_1,\lambda_2<0$ and $P(u,v)=0$.
Using the Br\'ezis-Lieb lemma\cite[Lemma 1.32]{willem}, we get
\begin{equation*}
  \|u_{n}\|_2^2=\|u_n-u_{a}\|_2^2+\|u_{a}\|_2^2+o_n(1),\quad\|v_{n}\|_2^2=\|v_n-v_{b}\|_2^2+\|v_{b}\|_2^2+o_n(1).
\end{equation*}
Let $a_1:=\|u_{a}\|_2>0$, $b_1:=\|v_{b}\|_2>0$, and $a_{1,n}:=\|u_n-u_{a}\|_2$, $b_{1,n}:=\|v_n-v_{b}\|_2$, then
$a^2=a_1^2+a_{1,n}^2+o_n(1)$ and $b^2=b_1^2+b_{1,n}^2+o_n(1)$.
On the one hand, using $(F_3)$, $P(u,v)=0$ and Fatou lemma, we have
\begin{align*}
\mathcal{J}(u)=&\mathcal{J}(u)-\frac{1}{2}P(u)=\frac{1}{2}\int_{\mathbb{R}}\big[(I_\mu*F(u,v))(u,v)\cdot\nabla F(u,v)-(3-\mu)(I_\mu*F(u,v))F(u,v)\big]\mathrm{d}x\\
\leq& \liminf_{n\to\infty} \frac{1}{2}\int_{\mathbb{R}}\big[(I_\mu*F(u,v))(u,v)\cdot\nabla F(u,v)-(3-\mu)(I_\mu*F(u,v))F(u,v)\big]\mathrm{d}x\\
=&\liminf_{n\to\infty}(\mathcal{J}(u_n,v_n)-\frac{1}{2}P(u_n,v_n))
=m(a,b).
\end{align*}
On the other hand, it follows from Lemma $\ref{6.1}$ that $\mathcal{J}(u,v)\geq m(a_1,b_1)\geq m(a,b).$ Thus $\mathcal{J}(u,v)= m(a_1,b_1)= m(a,b)$.
By Lemma \ref{6.3}, we obtain $a=a_1$ and $b=b_1$.
This implies $(u,v)$ is a ground state solution of problem $(\ref{problem})$-\eqref{problem'}.
\qed

\end{document}